\documentclass[a4paper,dvips]{article}
\usepackage{amsthm,amssymb,amsmath,enumerate,graphicx,epsf}

\newcommand{\COLORON}{1}
\newcommand{\NOTESON}{0}
\newcommand{\Debug}{0}
\usepackage[usenames,dvips]{color} 
\usepackage{amsthm,amssymb,amsmath,enumerate,graphicx,epsf,stmaryrd}

\newcommand{\comment}[1]{}
\newcommand{\COMMENT}[1]{}

\definecolor{darkgray}{rgb}{0.3,0.3,0.3}
\newcommand{\defi}[1]{{\color{darkgray}\emph{#1}}}

\comment{
	\begin{lemma}\label{}	
\end{lemma}
\begin{proof}

\end{proof}

\begin{theorem}\label{}
\end{theorem} 
\begin{proof} 	

\end{proof}

}

\newtheorem{proposition}{Proposition}[section]
\newtheorem{definition}[proposition]{Definition}
\newtheorem{theorem}[proposition]{Theorem}
\newtheorem{corollary}[proposition]{Corollary}

\newtheorem{lemma}[proposition]{Lemma}

\newtheorem{conjecture}{{Conjecture}}[section]

\newtheorem{problem}[conjecture]{{Problem}}

\newtheorem{examp}[proposition]{Example}

\newcommand{\FIG}{0}

\ifnum \NOTESON = 1 \newcommand{\note}[1]{ 

\hspace*{-30pt}
	{\color{blue}  NOTE: \color{Turquoise}{\small  \tt \begin{minipage}[c]{1.1\textwidth}  #1 \end{minipage} \ignorespacesafterend }} 
	
	}
\else \newcommand{\note}[1]{} \fi

\newcommand{\afsubm}[1]{ \ifnum \Debug = 1 {\mymargin{#1}}
\fi} 

\ifnum \Debug = 1 \newcommand{\sss}{\ensuremath{\color{red} \bowtie \bowtie \bowtie\ }}
\else \newcommand{\sss}{} \fi

\ifnum \FIG = 1 \newcommand{\fig}[1]{Figure ``{#1}''}
\else \newcommand{\fig}[1]{Figure~\ref{#1}} \fi

\ifnum \FIG = 1 
\else  \fi

\ifnum \Debug = 1 \usepackage[notref,notcite]{showkeys}
\fi

\ifnum \COLORON = 0 \renewcommand{\color}[1]{}
\fi

\newcommand{\showFig}[2]{
   \begin{figure}[htbp]
   \centering
   \noindent
   \epsfbox{#1.eps}
   \caption{\small #2}
   \label{#1}
   \end{figure}
}

\newcommand{\N}{\ensuremath{\mathbb N}}
\newcommand{\R}{\ensuremath{\mathbb R}}

\newcommand{\Q}{\ensuremath{\mathbb Q}}

\newcommand{\cb}{\ensuremath{\mathcal B}}
\newcommand{\cc}{\ensuremath{\mathcal C}}

\newcommand{\cf}{\ensuremath{\mathcal F}}

\newcommand{\ch}{\ensuremath{\mathcal H}}

\newcommand{\ck}{\ensuremath{\mathcal K}}

\newcommand{\cp}{\ensuremath{\mathcal P}}

\newcommand{\cw}{\ensuremath{\mathcal W}}

\newcommand{\oo}{\ensuremath{\omega}}

\newcommand{\del}{\ensuremath{\delta}}
\newcommand{\eps}{\ensuremath{\epsilon}}

\newcommand{\sig}{\ensuremath{\sigma}}

\newcommand{\sm}{\backslash}

\newcommand{\sydi}{\triangle}

\newcommand{\cls}[1]{\overline{#1}}

\newcommand{\nin}{\ensuremath{{n\in\N}}}

\newcommand{\unin}{\ensuremath{[0,1]}}

\newcommand{\sgl}[1]{\ensuremath{\{#1\}}}
\newcommand{\pth}[2]{\ensuremath{#1}\text{--}\ensuremath{#2}~path}
\newcommand{\edge}[2]{\ensuremath{#1}\text{--}\ensuremath{#2}~edge}

\newcommand{\arc}[2]{\ensuremath{#1}\text{--}\ensuremath{#2}~arc}

\newcommand{\seq}[1]{\ensuremath{(#1_n)_{n\in\N}}} 
 
\newcommand{\flo}[2]{\ensuremath{#1}\text{--}\ensuremath{#2}~flow} 

\newcommand{\g}{\ensuremath{G\ }}
\newcommand{\G}{\ensuremath{G}}

\newcommand{\ltp}{\ensuremath{|G|_\ell}}

\newcommand{\finl}{\ensuremath{\sum_{e \in E(G)} \ell(e) < \infty}}

\newcommand{\arE}{\vec{E}}

\newcommand{\BM}{Brownian Motion}

\newcommand{\Lr}[1]{Lemma~\ref{#1}}

\newcommand{\Tr}[1]{Theorem~\ref{#1}}

\newcommand{\Sr}[1]{Section~\ref{#1}}

\newcommand{\Prr}[1]{Pro\-position~\ref{#1}}

\newcommand{\Cr}[1]{Corollary~\ref{#1}}

\renewcommand{\iff}{if and only if}
\newcommand{\fe}{for every}
\newcommand{\Fe}{For every}

\newcommand{\st}{such that}

\newcommand{\ti}{there is}
\newcommand{\ta}{there are}

\newcommand{\obda}{without loss of generality}

\newcommand{\wrt}{with respect to}

\newcommand{\leth}{large enough that}

\newcommand{\HM}{Hausdorff measure}
\newcommand{\locon}{locally connected}

\newcommand{\HMT}{Hahn-Mazurkiewicz theorem}

\newcommand{\labtequ}[2]{
 \begin{equation} \label{#1} 	\begin{minipage}[c]{0.9\textwidth}  #2 \end{minipage} \ignorespacesafterend \end{equation} }

\newcommand{\mymargin}[1]{
  \marginpar{%
    \begin{minipage}{\marginparwidth}\small%
      \begin{flushleft}%
        {\color{blue}#1}%
      \end{flushleft}%
   \end{minipage}%
  }%
}%

\newcommand{\mySection}[2]{}

\newcommand{\LemArcC}{\cite[p.~208]{ElemTop}}
\newcommand{\LemArc}[1]{
	\begin{lemma}[\LemArcC]
	\label{#1}
	The image of a topological path with endpoints $x,y$ in a Hausdorff space $X$ contains an arc in $X$ between $x$ and $y$.
	\end{lemma}
}

\newtheorem{question}[conjecture]{{Question}}

\newcommand{\gl}{graph-like}
\newcommand{\Gl}{Graph-like}
\newcommand{\ugl}{uniformly graph-like}
\newcommand{\des}{disconnecting edge-set}
\newcommand{\gms}{geodesic metric space}
\newcommand{\pe}{pseudo-edge}
\newcommand{\arcon}{arcwise connected}
\newcommand{\smp}{strong Markov property}

\newcommand{\hg}{\ensuremath{\widehat{G}}}
\newcommand{\invg}{\ensuremath{\lim{G_i}}}
\newcommand{\Hm}{\ensuremath{\ch^1}}
\newcommand{\gseq}{graph approximation}
\newcommand{\finhm}{\ensuremath{\Hm(X)< \infty}}

\title{On graph-like continua of finite length}

\author{Agelos Georgakopoulos\thanks{Supported by EPSRC grant EP/L002787/1.}\medskip 
\\
  {\small Mathematics Institute, University of Warwick, CV4 7AL, UK}}

\begin{document}
\maketitle

\begin{abstract}
We extend the notion of effective resistance to metric spaces that are similar to graphs but can also be similar to fractals. Combined with other basic facts  proved in the paper, this lays the ground for a construction of \BM\ on such spaces completed in \cite{bmg}.
\end{abstract}

\section{Introduction}

A lot of recent work is devoted to extending fundamental theorems from finite graphs to infinite ones using topology; \cite{RDsBanffSurvey} gives a survey of some 40 papers of this kind. Some further research is emerging that uses this experience in order to extend such theorems to general topological spaces, usually continua \cite{lhom,ChrRiRoPla,ThomassenVellaContinua}. In this context, Thomassen and Vella \cite{ThomassenVellaContinua} introduce the notion of a \defi{\gl} space, defined as a topological space $X$ containing a set $E$ of pairwise disjoint copies of \R, called \defi{edges}, each of which is open in $X$ and has exactly two points in its frontier, \st\ the subspace $X\sm \bigcup E$ is  totally disconnected. We start this paper by observing that, although \gl\ spaces have till now mainly been thought of as generalizations of the Freudenthal compactification of a locally finite graph \cite{ThomassenVellaContinua}, there are more general interesting examples. 

\epsfxsize=.95\hsize
\showFig{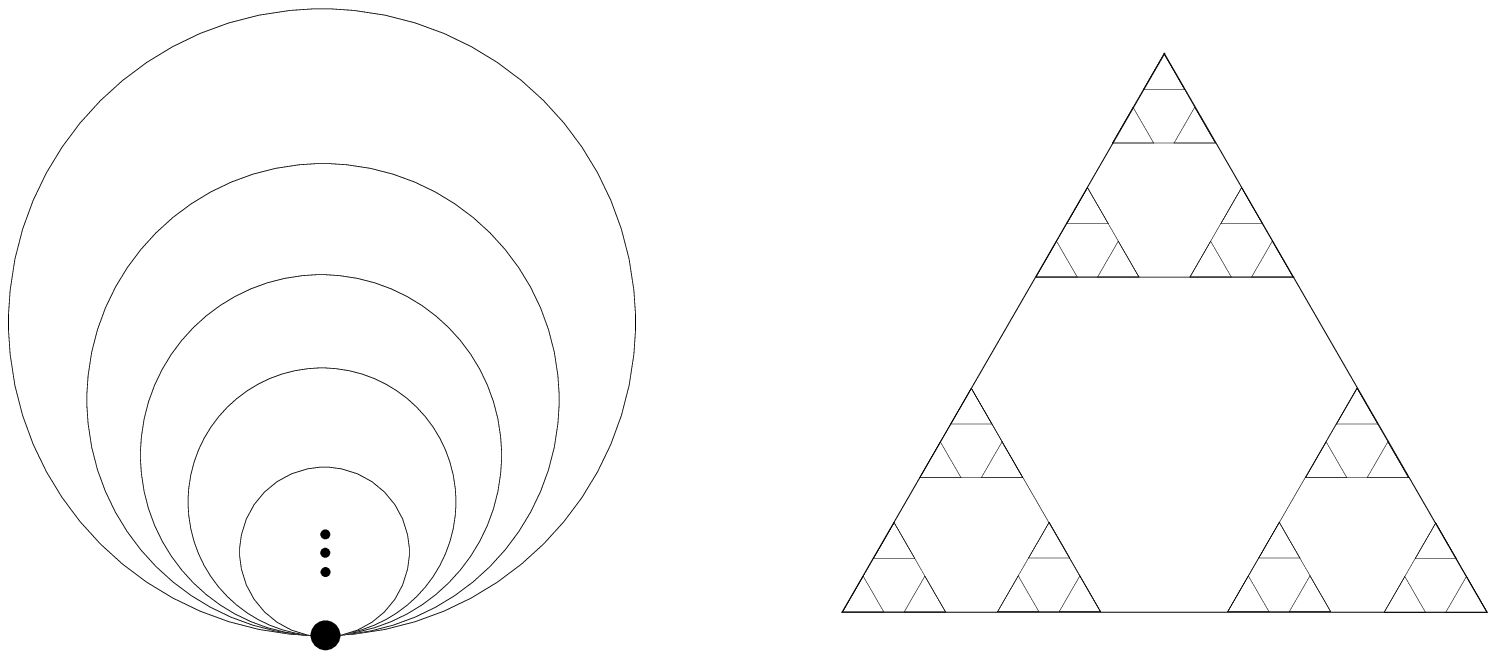}{Examples of \gl\ spaces.}

\fig{examples2} shows two such examples. The first one is the well-known Hawaian earring. The second is obtained from the Sierpinski gasket by replacing each articulation point with a copy of a real closed interval. The latter example suggests that many spaces that are not \gl\ can be deformed into a \gl\ space.

Several well-known graph-theoretic results have already been extended to \gl\ continua \cite{ChrRiRoPla,ThomassenVellaContinua}. 
In this paper we develop tools for studying analytic properties of these spaces. Most importantly, we lay the ground for a construction of \BM\ on \gl\ continua, which is completed in \cite{bmg}. 

Analysis on fractals has attracted a lot of research \cite{Kigami,KumRec,strichartz1999analysis}, the motivation coming both from pure mathematics and mathematical physics \cite{havlin_diffusion_1987,KumRec}. In this framework, \BM\ ---i.e.\ a stochastic process with continuous paths, the \smp, and further properties depending on the context---  plays an important role, but is much harder to construct than the classical Wiener process on \R. There are many  constructions of \BM\ on fractals, most of which are similar to the Sierpinski gasket  \cite{BaPeBro,KumBro,hambly_brownian_1997,Hattori,goldstein_random_1987,kusuoka_lecture_1993}. In comparison, we do not require our space to have any self-similarity or homogeneity properties, but we want it to have finite 1-dimensional \HM\ $\Hm(X)$.

Our first result is that every \gl\ continuum $X$ can be approximated by a sequence \seq{G} of finite graphs, which are topological subspaces of $X$, in the following sense.  
\Fe\ finite set of edges $F$ of $X$, and every component $C$ of $X\sm \bigcup F$, the graph $G_n\cap C$ is connected for almost all $n$ (\Tr{txseq}).  By metrizing them appropriately, and considering \BM\ $B_n$ on each member $G_n$, a \BM\ $B$ on $X$ is obtained in \cite{bmg} as a limit of the $B_n$. 
The hardest task is then to show that $B$ is uniquely determined by $X$ alone, and does not depend in particular on the choice of the sequence $(G_n)$. This involves proving a decomposition theorem for \gl\ spaces, which is one of the main results of this paper (\Sr{pedec}), and is only true when $\Hm(X)< \infty$.

\medskip
Using \gseq s\ \seq{G} as above, we show that if $X$ is a \gl\ continuum with $\Hm(X)< \infty$, then we can associate to any pair of points $p,q\in X$ an effective resistance $R(p,q)$ determined by $X$ alone:
\begin{theorem} \label{rconv}
Let $X$ be a \gl\ space with \finhm\ and let $\seq{G}$ be a \gseq\ of $X$. 
Then \fe\ two sequences \seq{p}, \seq{q} with $p_n,q_n \in G_n$, converging to points $p,q$ in $X$, 
the effective resistance $R_{G_n}(p_n,q_n)$ converges to a value $R(p,q)$ independent of the choice of the sequences $(p_n), (q_n)$ and $(G_n)$.
\end{theorem}
In fact, this $R(p,q)$ is even invariant under local isometry preserving the space (\Cr{lociso}).

\Tr{rconv} is important for the construction of \BM\ on $X$ in  \cite{bmg}, but it may find further applications in studying analytic properties of such spaces, using for example the methods of \cite{gfm} where effective resistance plays an important role.

\medskip
Answering a question of Menger \cite{MenUnt}, Bing \cite{BingPar} and Moise \cite{MoiGri} independently proved that every Peano continuum admits a compatible geodesic metric. \Gl\ continua are Peano by the \HMT\footnote{The \HMT\ \cite{Nadler} states that a metric space is a continuous image of the unit real interval if and only if it is compact, connected, and locally connected.}  and \Tr{glocon}. Thus the following result, which we prove in \Sr{gseqs} using the aforementioned \gseq s, can be viewed as a strengthening of the theorem of Bing and Moise in the \gl\ case.
\begin{theorem} \label{gmsI}
Every \gl\ continuum $X$ admits a compatible metric with respect to which $X$ is a \gms\ and $\Hm(X)$ is finite. 
\end{theorem}
This also suggests that the requirement of satisfying $\Hm(X)<\infty$ which we sometimes impose on our spaces is not very restrictive.

A related fact that we also prove (\Sr{secint}) is
\begin{corollary} \label{corintr}
Let $(X,d)$ be a  \gl\ continuum with $\Hm(X)<\infty$. Then the intrinsic metric  of $X$ is compatible with $d$, and turns $X$ into a \gms.
\end{corollary}

\section{Preliminaries}

\subsection{Standard definitions and facts}
The \defi{frontier} $\partial Y$ of a subspace $Y$ of a topological space $X$ is the set of points $p\in \overline{Y}$ \st\ every open neighbourhood of $p$ meets $X \sm Y$, where $\overline{Y}$ denotes the closure of  $Y$.

A \defi{continuum} is a compact, connected, non-empty metrizable space (some authors replace `metrizable' by Hausdorff).

An \defi{arc} $R$ is a topological space homeomorphic to the real interval $[0, 1]$. Its \defi{endpoints} are the images of $0$ and $1$ under any homeomorphism from \unin\ to $R$. The endpoints of $R$ will be denoted by $R^0,R^1$ whenever it does not matter which one is which. 

A \defi{topological path} in a space $X$ is a continuous map from a closed real interval to  $X$. 

\LemArc{patharc}

Let $\sigma: [a,b]\to X$ be a topological path in a metric space $(X,d)$. For a finite sequence $S=s_1, s_2, \ldots, s_k$ of points in $[a,b]$, let $\ell(S):= \sum_{1\leq i< k} d(\sigma(s_i),\sigma(s_{i+1}))$, and define the \defi{length} of $\sigma$ to be $\ell(\sigma):=\sup_S \ell(S)$, where the supremum is taken over 
all finite sequences $S=s_1, s_2, \ldots, s_k$ with $a=s_1<s_2< \ldots <s_k=b$. If $R$ is an arc in $(X,d)$, then we define its \defi{length $\ell(R)$} to be the length of a surjective topological path $\sigma:[0,1] \to R$ that is injective on $(0,1)$; it is easy to see that $\ell(R)$ does not depend on the choice of \sig\ \cite{Burago}.

The $n$-dimensional Hausdorff measure of a metric space $X$ is defined by 
$$\ch^n(X):= \lim_{\del \to 0} \ch^n_\del,$$ where
$$\ch^n_\del:=\inf \lbrace \sum_i  diam(U_i)^n \mid \bigcup_i U_i= X, diam(U_i) < \del \rbrace, $$
the infimum taken  over all countable covers \seq{U} of $X$ by sets $U_i$ of diameter less than \del.

It is well-known, and not hard to prove, that if $X$ is an arc then $\Hm(X)=\ell(X)$.

An $x$-$y$~geodesic in a metric space $(X,d)$ is a map $\tau$ from a closed interval $[0,l]\subset \R$ to $X$ such that $\tau(0)=x$, $\tau(l)=y$, and $d(\tau(t),\tau(t'))=|t-t'|$ for all $t,t' \in [0,l]$. If there is an $x$-$y$~geodesic \fe\ two points $x,y\in X$, then we call $X$ a \defi{\gms}, and $d$ a \defi{geodesic} metric.

\subsection{\Gl\ spaces} \label{sgl}

An \defi{edge} of a topological space $X$ is an open subspace $I\subseteq X$ homeomorhpic to the real interval $(0,1)$ \st\ the closure of $I$ in $X$ is homeomorphic to $[0,1]$. Note that the frontier of an edge consists of two points, which we call its \defi{endvertices} or \defi{endpoints}. An edge-set of a topological space $X$ is a set of pairwise disjoint edges of $X$. 
\labtequ{union}{If $E,F$ are finite edge-sets of $X$ then $\bigcup E\cap \bigcup F$ is also a finite edge-set, and $\bigcup E\cup \bigcup F$ is the union of a finite edge-set with a finite set of points.}

A topological space $X$ is \defi{\gl} if \ti\ an edge-set $E$ of $X$ such that $X\sm \bigcup E$ is totally disconnected. In that case, we call $E$ a \defi{\des}. A metric space $X$ is called \defi{\ugl}, if \fe\ \eps\ \ti\ a finite edge-set $S_\eps$ of $X$ \st\ the diameter of every 
component of $X \sm \bigcup S_\eps$ is less than \eps.

Every \ugl\ space is totally bounded, and so it is compact \iff\ it is complete. It follows easily from \eqref{union} that every \ugl\ space is \gl. We will show below (\Tr{glugl}) that the converse also holds for continua. 

We collect some basic facts about \gl\ spaces that will be useful later.

\begin{lemma}\label{count}	
Let $X$ be a graph-like continuum. Then every \des\ of $X$ is countable.
\end{lemma}
\begin{proof}
Let $E$ be a \des\ of $X$. It is not hard to prove   \cite{ThomassenVellaContinua}\footnote{See the claim in the proof of Theorem~2.1.} that \fe\ positive real \eps, the set $E_\eps$ of elements of $E$ with diameter larger than \eps\ is finite. Since $E= \bigcup_{\eps\in \Q}  E_\eps$, it follows that $E$ is countable.
\end{proof}

It might be possible to extend \Lr{count} to edge-sets that are not necessarily disconnecting by proving that  every edge-set of $X$ cointained in a \des, but we will not adress this question here.\sss

\medskip

We say that two points $w,y\in X$ are \defi{separated} by a set $S\subset X$, if \ta\ disjoint open sets $W,Y\subset X$ such that $w\in W$, $y\in Y$, and $X \subset W \cup Y $.

\begin{lemma}[{\cite[Theorem 2.5.]{ThomassenVellaContinua}}]\label{finsep}	
Let $X$ be a graph-like continuum. Then every two points of $X$ are separated by a finite edge-set.\footnote{\cite[Theorem 2.5.]{ThomassenVellaContinua}, asserts that  every two points of $X$ are separated by a finite set of points, but these points are chosen as interior points of edges in their proof.}
\end{lemma}

\begin{theorem} [{\cite[Theorem 2.1]{ThomassenVellaContinua}}]\label{glocon}
Every \gl\ continuum is locally connected.
\end{theorem} 

A consequence of this is
\begin{lemma}[{\cite[Corollary 2.4]{ThomassenVellaContinua}}]\label{cpc}	
Let $X$ be a graph-like continuum. Then every closed, connected subspace of $X$ is \arcon.
\end{lemma}

\begin{theorem}\label{glugl}
Let $X$ be a continuum. Then $X$ is \gl\ \iff\ $X$ is \ugl.
\end{theorem} 
\begin{proof} 	
The backward direction is easy as remarked above. For the forward direction, let $E$ be an edge-set of $X$ as in the definition of \gl, and recall that $E$ is countable by \Lr{count}. Fix an enumeration $e_1,e_2,\ldots$ of $E$, and let $E_n:= \{e_1,e_2,\ldots, e_n\}$. If $X$ is not \ugl, then for some $r>0$ and \fe\ $n$, \ta\ points $x_n,y_n$ in a common component of $X\sm \bigcup E_n$ \st\ $d(x_n,y_n)>r$. Let $x,y$ be accumulation points of the sequences $(x_n)$ and $(y_n)$ respectively, which exist since $X$ is compact, and note that $d(x,y)\geq r$; in particular, $x\neq y$.

By \Lr{finsep}, \ti\ an $m\in \N$ \st\ $E_m$ separates $x$ from $y$, which means that there are disjoint open sets $W\ni x$ and $Y\ni y$ \st\ $W \cup \bigcup E_n \cup Y = X$. Note that for some large enough $n_0>m$, all $x_n$ lie in $W$ for $n>n_0$ and all $y_n$ lie in $Y$. But this contradicts our assumption that $x_n,y_n$ lie in a common component of $X\sm \bigcup E_n$, since each component of $X\sm \bigcup E_n$ is contained in a component of $X\sm \bigcup E_m$.
\end{proof}

Using \Tr{glugl} we can simultaneously strengthen \Lr{finsep} and \Tr{glocon} as follows
\begin{corollary}\label{basis}
Let $X$ be a \gl\  continuum. Then the topology of $X$ has a basis consisting of connected open sets $O$ \st\ the frontier of $O$ is a finite set of points each contained in an edge.
\end{corollary} 
\begin{proof}
By \Tr{glugl}, $X$ is \ugl. The assertion now follows easily from the definition of \ugl: \fe\ $x\in X$ that is not on an edge (other points are easier to handle), let $C_\eps$ be a component of $X \sm \bigcup E$ of diameter at most $\eps$, where $E$ is a finite edge-set. By elementary topological arguments, it is possible to extend $C_\eps$ into an open set $O_\eps$ by uniting it with an open `interval' of diameter $\eps$ of each edge in $E$ that has an endvertex in $C_\eps$. As $\eps$ can be arbitrarily small, these sets $O_\eps$ form a basis.
\end{proof}

\medskip
The following basic fact was observed in \cite[Section~2]{ThomassenVellaContinua} in the case where $|E|=1$, but it is straightforward to extend to arbitrary finite $E$

\begin{lemma} \label{finE}
\Fe\ finite edge-set $E$ of a connected topological space $X$, the subspace $X \sm \bigcup E$ has only finitely many components, each of which is clopen in $X \sm \bigcup E$ and contains a point in $\overline{E}$.
\end{lemma}

\subsection{Metric graphs} \label{Megr}

In this paper, by a \defi{graph} \g we will mean a topological space homeomorhpic to a simplicial 1-complex. We assume that any graph \g is endowed with a fixed homeomorphism $h: K \to \G$ from a simplicial 1-complex $K$, and call the images under $h$ of the 0-simplices of $K$ the \defi{vertices} of \G, and the images under $h$ of the 1-simplices of $K$ the \defi{edges} of \G. Their sets are denoted by \defi{$V(G)$} and \defi{$E(G)$} respectively.  Most graphs considered will be \defi{finite}, that is, they will have finitely many vertices and edges.

A metric graph is a graph \g endowed with an assignment of lengths $\ell: E(G) \to \R_{>0}$ to its edges. This assignment naturally induces a metric $d_\ell$ on \g with the following properties. Edges are locally isometric to real intervals, their lengths (i.e.\ 1-dimensional \HM s) \wrt\ $d_\ell$ coincide with $\ell$, and \fe\ $x,y\in V(G)$ we have $d_\ell(x,y):= \inf_{P \text{ is an \arc{x}{y} }} \ell(P)$, where $\ell(P):= \sum_{P \supseteq e \in E(G)} \ell(e)$; see \cite{ltop} for details on $d_\ell$.

The length $\ell(G)$ of a metric graph \g is defined as $\sum_{e\in E(G)} \ell(e)$.

Note that (metric) graphs are \gl\ spaces, and their edges are also edges in the topological sense of \Sr{sgl}. 

An \defi{interval} of an edge $e$ of a graph or \gl\ space is a connected subspace of $e$.

\section{Approximating \gl\ continua by finite graphs} \label{gseqs}

In this section we show that \gl\ continua can be approximated well by a sequence of finite graphs; more precisely, we have
\begin{theorem} \label{txseq}
\Fe\ \gl\ continuum $X$, \ti\ a sequence of finite graphs \seq{G}, each contained in $X$, with the following property  
\labtequ{comps}{\fe\ finite edge-set $F$ of $X$ and every component $C$ of $X\sm F$, \ti\ a unique component of $G_n\sm F$ meeting $C$ for almost all $n$.}
\end{theorem}

This will play an important role in the rest of the paper. We will later also consider two possible ways of metrizing these graphs $G_n$, with various applications.

The sequence of finite metric graphs \seq{G} we will construct for the proof of \Tr{txseq} will have the additional property that each $G_i$ is contained in $G_{i+1}$ when seen as a topological space. In graph-theoretic terminology,  $G_i$ is a topological minor of $G_{i+1}$.

\begin{proof}[Proof of \Tr{txseq}]
Let $E$ be a \des\ of $X$, and recall that $E$ is countable (\Lr{count}), so let $E= \{e_1, e_2, \ldots\}$. Let $E_n:= \{e_1, \ldots, e_n\}$. 

To begin with, we construct $G_1$ as follows. If $X\sm e_1$ is disconnected ---in which case it has exactly two components, each containing an endvertex of $e_1$ by \Lr{finE}--- we let $G_1$ be the graph whose only edge is $e_1$ and whose only vertices are its endvertices $e_1^0,e_1^1$. If $X\sm e_1$ is connected, then we let $G_1$ have one more $e_1^0$-$e_1^1$~edge $f$ in addition to $e_1$: we let $f$ be any $e_1^0$-$e_1^1$~arc in $X\sm e_1$, which exists by \Lr{cpc}.

Then, for $i=2,3,\ldots$, we obtain $G_i$ from $G_{i-1}$ as follows. If $e_i$ is not a subarc of any edge of $G_{i-1}$, then it is disjoint to $G_{i-1}$ by construction; in this case, we add $e_i$ as an edge of $G_i$ and its endvertices, if not already present as vertices of $G_{i-1}$, as vertices of $G_i$. If $e_i$ is a subarc of some edge $f$ of $G_{i-1}$, then we subdivide $f$ into three or two edges by declaring the endvertices of $e_i$ to be vertices of $G_i$ (note that declaring a point to be a vertex has no effect on our graphs if we view them as topological spaces contained in $X$; it is only relevant when considering a graph as a discrete structure). Then, we add some further edges to $G_i$ if needed to make sure that the components of $X\sm E_i$ correspond one-to-one to the components of $G_i\sm E_i$: we go through all pairs of endvertices $y,z$ of edges in $E_i$ recursively, and if $y,z$ lie in the same component $C$ of $X\sm E_i$ but not in the same component of $G_i\sm E_i$, we choose a \arc{y}{z}\ $P$ in $C$, which exists by \Lr{cpc}. If $P$ only meets $G_i$ at $y,z$, then we add it to $G_i$ as a new edge, with endvertices $y,z$. If $P$ does meet $G_i\sm \{y,z\}$, then let $y',z'$ be the first and last point of $P\sm \{y,z\}$ in $G_i$ (which exist by elementary topological arguments), and add the inital and final subarcs $yPy'$ and $z'Pz$ to $G_i$ as new edges, with their endpoints as incident vertices. 
Note that these new edges do not meet any other edges of $G_i$, and any edge in $E \sm E_i$ is either disjoint to $G_i$ or contained in an edge of $G_i$.

This construction ensures that \fe\ $i$ and every component $C$ of $X\sm E_i$, the subspace $C \cap G_i$ is non-empty and connected. Since $G_j \supseteq G_i$ for $j>i$, we also obtain that $C \cap G_j$ is connected for $j>i$. Even more,  
we claim that the sequence $(G_i)_{i\in\N}$ has the desired property that \fe\ finite edge-set $F$, and every component $C$ of $X\sm F$, \ti\ a unique component of $G_i\sm F$ meeting $C$ for almost all $i$.

To see this, assume first that $F \subset E$. Now if we choose $\nin$ \leth\ $E_n$ contains $F$, then by the above remark, for every component $C'$ of $X\sm E_n$, the subspace $C' \cap G_i$ is non-empty and connected for $i\geq n$. Let $\cc$ be the (finite) set of components of $X\sm E_n$ contained in $C$ (since $F \subset E$, each component of $X\sm E_n$ is contained in a component of $X\sm F$). Construct an auxiliary graph $H$ with vertex set $\cc$ having an edge joining $C'$ to $C''$ whenever there is an edge in $E_n$ with endpoints in $C'$ and $C''$. Then $H$ is a connected graph because $C$ is a connected space. But as $G_i$ contains $E_n$, and $C' \cap G_i$ is connected \fe\ $C' \in \cc$, this implies that $C \cap G_i$ is connected as desired.

If $F$ is not a subset $E$, then each element of $F$ has an interval contained in some element of $E$ since $E$ disconnects $X$, and it is easy to see that the above arguments still apply. This completes our proof.
\end{proof}


Let $G\subseteq X:= \bigcup_n G_n$. Since each $G_i$ is contained in $G_{i+1}$ as a topological subspace, $G$ coincides with the set of points of $X$ that appear in almost every $G_i$. 

We will now consider two metrizations of the $G_n$, leading to corresponding metrics of $G$. The first one is obtained by assigning to each edge $e$ of $G_n$ a length $\ell_X(e)$ equal to the length of $e$ as an arc in $X$ (recall that $G_n\subseteq X$). Let $d^\ell_n$ denote the induced metric on the metric graph $(G_n, \ell_X)$.
Since $G_i$ is contained in $G_{i+1}$, it follows that \fe\ $y,z\in G$, their distance $d^\ell_n(y,z)$ is monotone decreasing with $n$. This allows us to define the metric $d^\ell$ on $G$ by $d^\ell(y,z):= \lim_i d^\ell_i(y,z)$. We let $\widehat{G^\ell}$ denote the completion of $(G,d^\ell)$. We will prove below that $\widehat{G^\ell}$ is homeomorphic to $X$ when \Hm(G) is finite (\Tr{Xhg}).

To obtain the second metrization, we first fix an assignment of lengths $\ell'_f: E \to \R_{>0}$ to the \des\ $E$ used in the construction of $(G_n)$ such that $\sum_{e\in E} \ell'_f(e)=L< \infty$, and then we assign to each edge $e$ of $G_n$ a length $\ell_f(e)$ equal to the sum of the lengths of the elements of $E$ contained in $e$ (thus if $e$ happens to be an element of $E$, then we have $\ell_f(e)= \ell'_f(e)$). Let $d^f_n$ denote the induced metric on the metric graph $(G_n, \ell_f)$. Again, $d^f_n(y,z)$ is monotone decreasing \fe\ $y,z\in G$, and this induces a metric $d^f(y,z):= \lim_i d^f_i(y,z)$ on $G$. The corresponding completion will be denoted by $\widehat{G^f}$. We will show that $\widehat{G^f}$ is always homeomorphic to $X$, and that $\widehat{G^f}$ is a \gms\ (\Tr{gms}).

\begin{theorem}\label{Xhg}
Let $(X,d)$ be a  \gl\ continuum with $\Hm(X)<\infty$. Then 
$X$ is homeomorphic to $\widehat{G^\ell}$.
\end{theorem} 
\begin{proof} 

In order to be able to extend the identity map from $G$ to $X$ into a homeomorphism $h: \widehat{G^\ell} \to X$,
it suffices to prove that a sequence $\seq{p}$ of points of \g is Cauchy \wrt\ $d^\ell$ \iff\ it is Cauchy \wrt\ $d$. 

The forward direction follows from the fact that $d^\ell\geq d$ because the length of any arc in a metric space is at least the distance of its endpoints. 

For the backward direction, let $\seq{p}$ be Cauchy \wrt\ $d$, and let $p$ be its limit in $X$, which exists since $X$ is compact, hence complete. Since $\Hm(X)<\infty$, any sequence $\seq{O}$ of open neighbourhoods of $p$ with $\lim diam(O_n)=0$ satisfies $\lim \Hm(O_n)= 0$, because $\Hm(X) = \lim \Hm(X \sm O_n)$ by the definition of $\Hm$ and $\Hm(X) \leq \Hm(X \sm O_n) + \Hm(O_n)$.

Combining this with \Cr{basis}, we can find, \fe\ $\eps>0$, a connected open neighbourhood $O$ of $p$ whose frontier is contained in a finite edge-set and $\Hm(O)<\eps$. We claim that for any $p_i,p_j \in O$, we have $d^\ell(p_i,p_j)\leq \Hm(O_n)<\eps$. Since almost all $p_i$ lie in this neighbourhood $O$ of $p$, we can conclude that $(p_i)$ is then Cauchy \wrt\ $d^\ell$.  

To prove the above claim, note that $O_n\cap G_n$ is connected for $n$ large enough by \eqref{comps} and \Prr{dense}. Thus $G_n$ contains an \arc{p_i}{p_j} $A$ contained in $O$. By the definition of $d^\ell$, we have $d^\ell(p_i,p_j)\leq \ell(A)$, and since $A\subseteq O$ we have $\ell(A) \leq  \Hm(O)$, which proves our claim. 


\end{proof}

\note{\seq{G}\ convergence in the Gromov-Hausdorff sense to the completion of $\bigcup G_n$. To show the Gromov-Hausdorff convergence of \seq{G}, by \cite[Corollary 7.3.28]{Burago}, it suffices to show that \fe\ $\eps$ \ti\ $\nin$ \st\ \fe\ $i>n$ \ti\ an $\eps$-isometry from $G_i$ to $\widehat{\bigcup G_n}$.
}

The completion \wrt\ our other metric on \g is homeomorhpic to $X$ in greater generality, and this implies \Tr{gmsI} from the Introduction.

\begin{theorem}\label{gms}
Let $(X,d)$ be a  \gl\ continuum. Then 
$\widehat{G^f}$ is a \gms\ homeomorphic to $X$. 
\end{theorem} 
\begin{proof} 
We can prove that $\widehat{G^f}$ is canonically homeomorhpic to $X$  similarly to the proof of \Tr{Xhg}: the fact that a sequence $\seq{p}$ of points of \g is Cauchy \wrt\ $d^f$ if it is Cauchy \wrt\ $d$ can be proved with similar arguments, except that rather than using any condition on $\Hm$ we observe that we can find a connected open neighbourhood $O$ of $p$ whose frontier is contained in a finite edge-set such that $\sum_{e\in E \cap O} \ell_f(e)$ is arbitrarily small. 

For the converse statement, let  $\seq{p}$ be a Cauchy sequence in $\widehat{G^f}$. We want to prove that it is also Cauchy in $X$. Easily, we may assume \obda\ that no edge of $X$ contains infinitely many elements of $\seq{p}$. Let $O$ be a connected open neighbourhood  of $p:= \lim p_n$ in $X$ whose frontier is contained in a finite edge-set $F$. Let $F':= F\cap (X \sm O)$, and notice that $F'$ is an edge-set comprising an interval of each element of $F$. Then by the definition of $G_n$ and $d_n^f$, if $p_i$ lies outside $O \cup F$ then we have $d_n^f(p_i,p)\geq \min_{e\in F'} \ell_f(e)$ because every path in $G_n$ joining a point outside $O \cup F$ to a point in $O$ has to go through some edge in $F'$. Since the $p_n$ converge to $p$ \wrt\ $d^f$, and since we are assuming that no edge contains infinitely many of them, it follows that almost all $p_n$ lie in $O$. Choosing a sequence of such $O$ converging to $p$ now implies that $\seq{p}$ is Cauchy in $X$ as desired.
\medskip

The fact that $\widehat{G^f}$ is a \gms\ follows from the fact that $\widehat{G^f}$ is compact (since it is homeomorhpic to $X$) and it is a length space by definition, i.e.\ the distance between any two points equals the infimum of the lengths of the paths joining them.

\comment{
	We define \seq{G}\ and \hg\ as above, whereby if we are given a \des\ $E$ of finite length we use that edge-set in the construction of \seq{G}. 
It is easy to see that \hg\ is a length space by the definition of its metric. Since it is homeomorphic to the compact space $X$ by \Tr{Xhg}, \hg\ is a \gms. 

To see that \hg\ has finite length...

	The second sentence is clearly true 
	by the construction of \hg. 
}
\end{proof}

In \Tr{gms}, if $d$ is a metric of $X$ \wrt\ which $X$ has a \des\ $E$ of finite length, then \ti\ a compatible geodesic metric $d'$ of $X$ such that each edge of $E$ has the same length in $d$ and $d'$.

Graph sequences as in \Tr{txseq} have the following additional property

\begin{proposition} \label{dense}
Let $X$ be a \gl\ space and \seq{G} a sequence of graphs as in  \Tr{txseq}. Then $\bigcup G_n$ contains every edge of $X$ and is dense in $X$.
\end{proposition}
\begin{proof} 
Let $e$ be an  edge of $X$ and suppose that $\bigcup G_n$ misses a point $p\in e$. Consider two disjoint edges $f_1,f_2$ contained in $e$, each having $p$ as an endpoint. Then letting $F:=\{ f_1,f_2\}$ contradicts \eqref{comps} since $p$ forms a component of $X\sm F$.

As any \des\ $E$ of $X$ is dense in $X$ by definition, and $\bigcup G_n$ contains $E$ as we just saw, $\bigcup G_n$ is dense in $X$.
\end{proof}

In fact we can strengthen \Prr{dense} a bit. We say that an edge $e$ of $X$ is a \defi{pending edge}, if at least one component of $X\sm e$ is a singleton (which must be an endpoint of $e$). Then for every edge $e$ that is not a pending edge, some $G_n$ contains $e$.

\section{Lengths and Hausdorff measure} \label{ellH}

Define a \defi{\gseq} of $X$ to be a sequence $\seq{G}$ of finite subgraphs of $X$ satisfying \eqref{comps}, \st\ \fe\ edge $e\in E(G_n)$ the length $\ell(e)$ of $e$ in $G_n$ coincides with the length of the corresponding arc of $X$. We established the existence of such sequences in \Tr{txseq}. In this section we show that they approximate $X$ well also in terms of the Hausdorff measure. This fact is a key tool in the proof of our decomposition theorem in the next section.

Recall that the $n$-dimensional Hausdorff measure of a metric space $X$ is defined by 
$$\ch^n(X):= \lim_{\del \to 0} \ch^n_\del,$$ where
$$\ch^n_\del:=\inf \lbrace \sum_i  diam(U_i)^n \mid \bigcup_i U_i= X, diam(U_i) < \del \rbrace. $$
We introduce a quantity $\ch^G_\del$ similar to $\ch^1_\del$, 
that will be useful in \Sr{pedec}: given a \des\ $E$ of $X$, let
$$\ch^G_\del:= \ell(E_\del) + \sum_{\stackrel{K \text{ is a component} }{\text{ of } X\sm E_\del}} diam(K),$$
where $E_\del \subset E$ is a finite edge-set chosen so that every component $K$ in the sum has diameter $diam(K)< \del$. Such a choice is possible by \Tr{glugl}.

\begin{theorem}\label{HmeqL}	
Let $X$ be a \gl\ continuum and \seq{G} a \gseq\ of $X$. Then $\lim \ell(G_n) = \Hm(X) = \lim_{\del \to 0} \ch^G_\del$.  (In particular, the latter limit exists.)
\end{theorem}
\begin{proof}

We claim that the following inequalities hold, from which the assertion follows
$$\lim \ell(G_n) \geq \limsup \ch^G_\del \geq \liminf \ch^G_\del \geq \Hm(X) \geq \lim \ell(G_n).$$

Let us first show that $\lim \ell(G_n) \geq \ch^G_\del$ \fe\ \del. For this, fix \del\ and let $K$ be a component of $X\sm E_\del$ as in the definition of $\ch^G_\del$. If $n$ is large enough, then by \Prr{dense}, $G_n$ contains almost all of $\bigcup E_\del$, as well as a pair of points $x, y \in K$ with $d_X(x,y)$ arbitrarily close to $diam(K)$. Moreover, by \eqref{comps} we know that $G_n \cap K$ is connected, and so \ti\ a path $P_K$ in $K \cap G_n$ joining $x$ to $y$. Thus, denoting the set of  components of $X\sm E_\del$ by $\ck$, we have
$$\sum_{K\in \ck} \ell(P_K) \geq \sum_{K\in \ck}  diam(K) - \eps$$ 
for an arbitrarily small \eps. Since $G_n$ contains all the (pairwise disjoint) paths $P_K$ as well as most of $\bigcup E_\del$, we have 
$$\ell(G_n) \geq  \ell(E_\del) + \sum_{K\in \ck} diam(K) - \eps = \ch^G_\del- \eps$$
 and the inequality follows.

To see that $\liminf \ch^G_\del \geq \Hm(G)$, notice that $E_\del$ can be covered by a set of balls $\cw$ each of diameter less than \del\ with total diameter $\sum_{W\in \cw} diam(W) \leq \ell(E_\del)$ by the definition of 
$\ell(E_\del)$, and that the union of $\cw$ with the set of  components of $X\sm E_\del$ appearing in the definition of $\ch^G_\del$ is a candidate for the cover \seq{U} in the definition of $\ch^n_\del$.

Finally, the inequality $\Hm(G) \geq \lim \ell(G_n)$ is an easy consequence of the definitions and the fact that every edge of $G_n$ has, by definition, the same length as an arc in $X$.

\end{proof}

As a corollary, we obtain that both $\lim_{\del \to 0} \ch^G_\del$ and $\lim \ell(G_n)$ is independent of the choice of the \des\ $E$.

\section{The \pe\ decomposition theorem} \label{pedec}

In this section we formulate and prove a decomposition theorem for \gl\ continua of finite \Hm\ (\Tr{struct}) that will be useful in the proof of \Tr{rconv}, \Cr{corintr}, and the construction of \BM\ in \cite{bmg}. This decomposition will be based on the following notion:
\begin{definition}
A \defi{\pe} of a metric space $X$ is an open connected subspace $f$ \st\ $|\partial f|=2$ and no homeomorphic copy of the interval $(0,1)$ contained in $\cls{f}$ contains a point in $\partial f$. We denote the elements of $\partial f$ by $f^0,f^1$, and call them the \defi{endpoints} of $f$. Note that every edge is a \pe.
\end{definition}

{\em Example 1:} Every two distinct points of an \R-tree are the endpoints of a (unique) \pe.

\medskip
{\em Example 2:} Start with the unit real interval $I$, and let $D \subset I$ be the complement of a Cantor set of positive Lebesque measure, i.e.\ $D$ is a set of disjoint open intervals of total length less than 1, say 1/2. Now for every such interval $J$, add to the space a copy $J'$ of $J$ so that $J$ and $J'$ have the same endpoints, to obtain a space $X$ (\fig{CantorPE}).

\showFig{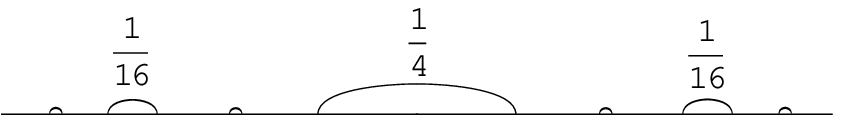}{A \gl\ \pe\ the \HM\ of which does not coincide with the sum of the lengths of its edges.}

Note that $X$ is graph-like (take the above intervals and their copies as the edges), and it is a \pe. An interesting fact about $X$ is that $\Hm(X)= 3/2$, although the sum of the lengths of its edges is 1. The reason is that a positive proportion of $\Hm(X)$ lies in the complement of the edges, which is a Cantor set.
\bigskip

\comment{
\begin{lemma}\label{}	
Every \pe\ of a \gl\ continuum $X$ is \locon.
\end{lemma}
\begin{proof}
By \Tr{glocon}, $X$ is \locon. 
It follows immediately from the definition of local connectedness that every open connected subspace of a \locon\ space is itself \locon. As \pe s have these properties, they are \locon. (We could have let $X$ be an arbitrary \locon\ metric space.)
\end{proof}
}

By \Lr{cpc}, \fe\ \pe\ $f$, there is an \arc{f^0}{f^1} $P_f$ in $f$. Note that the following inequalities hold:
\labtequ{Hld}{$\Hm(f) \geq \ell(P_f) \geq d(f^0,f^1).$}
We define the \defi{discrepancy} $\delta(f)$ of a \pe\ $f$ by  $\delta(f):=\Hm(f) - d(f^0,f^1)$, which by the above inequality is always non-negative. 

The main result of this section is the following decomposition theorem for \gl\ spaces of finite $\Hm$
\begin{theorem} \label{struct}
\Fe\ \gl\ continuum $X$ with \finhm, there is a set $\cf$ of pairwise disjoint \pe s of $X$ with $\sum_{f\in \cf} \Hm(f) = \Hm(X)$. Moreover, \fe\ $\eps>0$ we can choose $\cf$ so that $\sum_{f\in \cf} \delta(f)< \eps$.
\end{theorem}

Example 2 above shows why this assertion becomes false if we replace the word `\pe s' by the word `edges'. It also gives a lot of insight into its involved proof. 

Before proving \Tr{struct} we will collect some more simpler facts about \pe s that will also be useful later. Recall the definition of a \gseq\ from \Sr{ellH}.


\begin{lemma}\label{pseudo1}	
\Fe\ \pe\ $f$ of a \gl\ space $X$ and every \gseq\ $\seq{G}$ of $X$, almost every $G_n$ contains an \arc{f^0}{f^1} unless one of $f^0,f^1$ has an open neighbourhood contained in $f$.
\end{lemma}
\begin{proof}
By \Lr{finsep} we can find a finite edge-set $E$ separating $f^0$ from $f^1$. If the component $C^0$ of $X \sm \bigcup E$ containing $f^0$ is contained in $f$ we are done, and similarly with the component $C^1$ containing $f^1$. Otherwise, by choosing $E$ larger if needed, we may assume that for each $i=0,1$ there is an edge $e_i$ in $E$ with an endpoint in $C^i \sm f$.

Moreover, we may assume that each edge in $E$ has an endpoint in either $C^0$ or $C^1$, for all other edges can  be removed from $E$ without losing any of the above properties.
We claim that \ti\ a  component  $C$ of $X \sm \bigcup E$ contained in $f$ and containing enpdoints $p,q$ of edges $e_3,e_4 \in E$ whose other endpoints lie in $C^0$ and $C^1$ respectively. Indeed, let $\cf$ be the set of components of $X \sm \bigcup E$ meeting $f$ excluding $C^0$ and $C^1$, and note that since $f$ is connected, each element of $\cf$ is contained in $f$. By \Lr{finE}, each element $F$ of $\cf$ is open in $X \sm \bigcup E$ and contains an endpoint of an edge $e_F$ in $E$. By the remark above, the other endpoint lies in either $C^0$ or $C^1$. Let $\cf^0$ be the set of those $F\in \cf$ \st\ \ti\ an edge $e_F$ as above with an endpoint in $C^0$, and define $\cf^1$ similarly. Now if our claim is false, then $\cf^0 \cap \cf^1 = \emptyset$; but then, we can find two disjoint open sets separating $f$ contradicting its connectedness: let $O_1:= C^0 \cup \bigcup \cf^0 \cup \{e\in E \mid e \text{ has an endpoint in } C^0$ and define $O_2$ similarly.

Now by \eqref{comps}, the intersection of $G_n$ with each of $C^0, C$ and $C^1$ is connected for large enough $n$. Moreover, by \Prr{dense}, $G_n$ contains an interval of each element of $E$ for large enough $n$. Thus $G_n$ eventually contains a path from $e_0$ to $e_3$ in $C^0$, a path from $e_3$ to $e_4$ in $C$ and a path from $e_4$ to $e_1$ in $C^1$. Concatenating these three paths with the edges $e_3, e_4$, we obtain an arc of $X$ that starts and finishes outside $f$ but meets $f$. Since $f^0,f^1$ separate $f$ from the rest of $X$, this arc must contain an \arc{f^0}{f^1}.
 
\medskip

Note that if $f^1$ does have an open neighbourhood contained in $f$, then we cannot guarantee that any $G_n$ contains an \arc{f^0}{f^1}, but we can guarantee that \fe\ $\eps>0$, $G_n$ eventually contains an arc from $f_0$ to a point \eps-close to $f^1$ by the same arguments.
 

\comment{
\begin{enumerate}
\item there is only one edge in $E^0_n$
\item there are at least two edges in $E^0_n$, but all of them are contained in $f$
\item $E^0_n$ contains edges both inside and outside $f$
\end{enumerate}
}

\end{proof}

Many of the properties of edges are preserved by \pe s as may already have become apparent. Here way observe some more that will be usefull later. Note that \Lr{finE} remains true if one replaces the edge-set $E$ by a set of \pe s: the only properties of an edge used in its proof were the fact that edges are open, and their frontier comprises two points, and this is also true for \pe s. We repeat that lemma with edges replaced by \pe s

\begin{lemma} \label{finEpe}
\Fe\ finite set $E$ of pairwise disjoint \pe s of a connected topological space $X$, the subspace $X \sm \bigcup E$ has only finitely many components, each of which is clopen in $X \sm \bigcup E$ and contains a point in $\overline{E}$.
\end{lemma}

Next, we extend \eqref{comps} to \pe s:

\labtequ{comps2}{\Fe\ finite set $F$ of pairwise disjoint \pe s, every \gseq\ \seq{G}, and every component $C$ of $X\sm F$, \ti\ a unique component of $G_i\sm F$ meeting $C$ for almost all $i$.}
To prove this, note that $\partial C$ is contained in $\bigcup_{f\in F} \partial f$, in particular it is finite. For $p\in \partial C$, we let $E_p$ be a finite edge-set separating $p$ from the other endpoints of all \pe s\ $f\in F$ with $p\in \partial f$ (\ti\ at least one such $f$ by the previous remark, but there may be several); such an edge-set exists by \Lr{finsep}.

Let $E:= \bigcup_{p\in \partial C} E_p$, and let $E':= E \cap (X\sm C)$. Note that $E'$ is a finite edge-set still separating any $p\in \partial C$ from the other endpoints of all \pe s\ $f\in F$ with $p\in \partial f$. Since $E' \cap C = \emptyset$ by definition, $C$ is contained in a component $C'$ of $X \sm E'$. We can apply \eqref{comps} to $C'$ to deduce that for almost every $G_i$, the subgraph $G_i \cap C'$ is connected. We claim that $G_i \cap C$ is also connected for every such $i$. Indeed, each $p\in \partial C$ separates any \pe\ $f$ with $p\in \partial f$ from $C$, and so no arc of $G_i\cap C'$ connecting two points of $C$ can visit $f$, which implies that $(G_i\cap C') \sm F = G_i \cap C$ is connected as claimed.

\begin{lemma}\label{Sarcs}	
Let $S$ be a finite set of points of a \gl\ continuum $X$. Then for every component $C$ of $X \sm S$ and every $p\in \partial C$, the subspace $C\cup \sgl{p}$ is \arcon.
\end{lemma}
\begin{proof}
We claim that $\overline{C}$ is a \gl\ continuum. Indeed, $\overline{C}$ is compact since it is a closed subset of the compact space $X$. Moreover, it is \gl, since the intersection of a \des\ of $X$ with $C$ is clearly a \des\ of $\overline{C}$. 

Let $E$ be a \des\ of $\overline{C}$, and \seq{E} an increasing sequence of finite subsets of $E$ \st\ $\bigcup E_n= E$.
Define an auxiliary graph $G_n$ as follows. The vertices of $G_n$ are the components of $\overline{C} \sm E_n$, and its edges are the edges in $E_n$, with $e\in E_n$ being incident to a vertex $x$ of $G_n$ whenever $e$ has an endpoint in the closure of the component $x$. It is not hard to see that $G_n$ is  connected since $X$ is, for the set of point inside any component of $G_n$ forms a component of $X$. We denote the vertex of $G_n$ containing a point $x\in \overline{C}$ by $x^*$.
Let $G'_n$ denote the graph obtained from $G_n$ by deleting all the vertices corresponding to the components containing a point in $\partial C \sm \sgl{p}$. Let $x$ be a point of $C$, for which we would like to find an \arc{x}{p}. We distinguish two cases.

The first case is when for some $n$, the graph $G'_n$ has a \pth{x^*}{p^*}. Then, if $P= v_0(=x^*) e_0 v_1 e_1 \ldots e_k v_{k+1}(=p^*)$ is such a path, we will transform it into the desired \arc{x}{p}. For this, we apply \Lr{cpc} to each component $v_i$ of $\overline{C} \sm E_n$ appearing in $P$ to obtain an arc in $v_i$ joining the endpoint of $e_i$ to the endpoint of $e_{i+1}$ in $v_i$, unless $v_i$ is $x^*$ or $p^*$ in which case the arc joins $x$ to the endpoint of $e_0$ or the endpoint of $e_k$ to $p$ respectively. Concatenating these arcs in the right order with the edges $e_i$ in $P$, we obtain the desired \arc{x}{p}.

The second case is when $G'_n$ has no such path \fe\ $n$. In this case we will obtain a contradiction to the connectedness of $C$. For this, let $C_p$ be the connected component of $G'_n$ containing $p^*$. Let $U_n\subset C$ denote the union of all vertices and edges in  $C_p$, and note that $U_n$ is an open subspace of $C$. Similarly, the union $V_n$ of all remaining vertices and edges of $G'_n$ is open. The sets $U_n,V_n$ are disjoint, but the do not yet separate $C$ as they miss the components $b^*$ for $b\in \partial C$. However, both $U_n,V_n$ are increasing in $n$ with respect to set inclusion, because every component of $\overline{C} \sm E_{n+1}$ is contained in a component of  $\overline{C} \sm E_{n+1}$, and so each vertex of $G_n$ decomposes into a connected subgraph of $G_{n+1}$ by an argument similar to that used to prove that $G_n$ is connected. Thus the subspaces $\bigcup U_n, \bigcup V_n$ of $C$ are open, disjoint, and their union is $C$, which contradicts the connectedness of $C$ proving that this case cannot occur.

\end{proof}

We can now prove the main result of this section.
\begin{proof}[Proof of \Tr{struct}]
Let $E$ be a finite edge-set of $X$, and consider a component $K$ of $X \sm \bigcup E$. Let $x,y$ be points of $K$ with $d(x,y)=diam(K)$ (which exist since $K$ is closed by \Lr{finE}). 
As $K$ is \arcon\ by \Lr{cpc}, \ti\ an \arc{x}{y} $\cp$ in $K$, as well as a \arc{q}{\cp} $A_q$ \fe\ 
$q\in \partial K\subset \partial E$. Let $q'\in \cp$ be the first point of $A_q$ on $\cp$; there might be several candidates for $q'$ as such an arc is not necessarily unique, but we just choose one of them. 

\medskip
\fig{figpes} shows how the situation could look like. Notice that the set $\Pi$ of bold vertices in that figure delimits a set of  \pe s. If $K$ is a finite graph, then it is not so hard to find such a set $\Pi$, but for a general \gl\ space we need to work harder. 
\epsfxsize=\hsize

\showFig{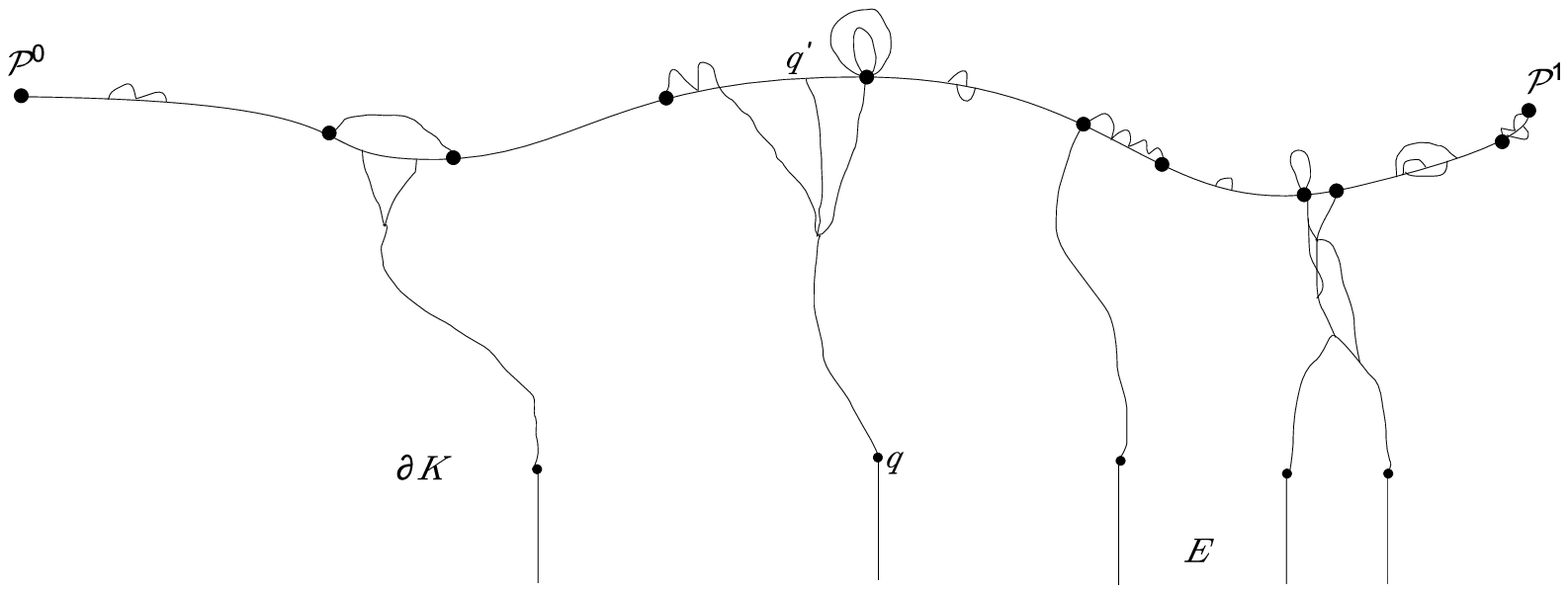}{The component $K$. Bold points indicate elements of $\Pi$.}

Note that if $diam(K)$ is close to $\Hm(K)$, which can be achieved using \Tr{HmeqL}, then $\cp$ will bear most of the measure $\Hm(K)$, and the rest of the graph in \fig{figpes} will have relatively short length. This can be used in order to prove that most of the length of $\cp$ lies in the aforementioned \pe s, and we will use this idea in our proof.

\medskip
Back to the general case, call a (closed) subarc $I$ of $\cp$ a \defi{bridged} subarc, if \ti\ a circle $S$ in $K$ \st\ $I= S\cap \cp$; 
we call $S \sm I$ the corresponding \defi{bridge} (\fig{figpes} shows several bridges). For each of the points $q'$ chosen above, we also declare $\sgl{q'}$ to be a \defi{bridged} subarc of $\cp$, and we declare each of the two endpoints $\cp^0,\cp^1$ of $\cp$ to be bridged as well.  Let $\cb$ be the set of bridged subarcs of $\cp$.


We call a closed subarc $I$ of $\cp$ \defi{super-bridged}, if all but countably many points of $I$ lie in $\bigcup \cb$ (our trivial bridged subarcs $q'$ and $\cp^0,\cp^1$ are also allowed as super-bridged).
We claim that 
\labtequ{maxsb}{each point $p\in \overline{\bigcup \cb}$ lies in a maximal super-bridged subarc of $\cp$.}
%

For this, let $D$ be a countable dense subset of $\cp$ containing the (finite) set $\{q'\mid q\in \partial K\} \cup \{\cp^0,\cp^1\}$.
\Fe\ $x,y\in D$, we define the subset $$\cp_{x,y}:= \overline{\bigcup_{b \in \cb, b\cap [x,y] \neq \emptyset} b}$$ of $\cp$, where $[x,y]$ denotes the subarc of $\cp$ bounded by $x,y$.
Let $Q$ be the set of all the $\cp_{x,y}$ that are super-bridged (hence connected) and contain $p$. 

We claim that $Q$ is non-empty. Indeed, recall that \ti\ at least one bridged arc $b$ containing $p$.
If $b$ is non-trivial, then it contains points $x,y$ of our dense subset $D$ of $\cp$, and then $\cp_{x,y}$ contains $p$ and is connected, and hence it lies in $Q$. If $b$ is trivial, then its unique point $p$ coincides, by the definition of $\cb$, with one of the points in $\{q'\mid q\in \partial K\} \cup \{\cp^0,\cp^1\}$. Since we chose $D$ so as to contain these points, it is now easy to check that $\cp_{p,p}\in Q$ in this case.

We claim that $I:= \overline{\bigcup Q}$ is a maximal super-bridged subarc of $\cp$. Easily, $\bigcup Q$, and hence $I$, is connected since it is the union of connected sets with a common point $p$. To see that $I$ is super-bridged, note that any point in its interior that is not in $\bigcup \cb$ lies in some $\cp_{a,b}$. Since there are only countably many such sets, and each of them, being super-bridged, contains at most countably many such points, there are at most countably many such points in $I$ as desired. It remains to check that $I$ is maximal with these properties. If not, then \ti\ a $d\in \cb$ such that $d \sm I$ is non-empty and $I\cup d$ is connected. Thus \ti\ $\cp_{a,b}\in Q$ \st\ $\cp_{a,b}\cup d$ is connected.
Then, as $d \sm I$ is non-empty, $d$ starts before $\cp_{a,b}$ or ends outside $\cp_{a,b}$ (or both).
In the former case, the set  $J:=\cp_{d,b}$ is super-bridged because it is the union of $\cp_{a,b}$ with $d$. In the latter case, $J:=\cp_{a,d}$ is super-bridged because it is the union of $\cp_{a,b}$ with some arcs starting inside $\cp_{a,b}$. In both cases, $J$ is contained in $Q$, and hence in $I$. But as $d\subseteq J$ by definition, this contradicts the assumption that $d \sm I$ is non-empty, proving that $I$ is the subarc sought after by \eqref{maxsb}.

\medskip


Since \fe\ $q\in \partial K$ we have $q'\in \cb$ by the definition of $\cb$, \eqref{maxsb} implies that $q'$ lies in a maximal super-bridged subarc $P_q$ \fe\ such $q$ (we might have $P_q=P_r$ for $q\neq r \in {\partial} K$ though). Similarly, each $q\in \{\cp^0,\cp^1\}$, \ti\ a maximal super-bridged subarc $P_q$ containing $q$. Let $\Pi$ be the set of endpoints of the $P_q$ \fe\ $q\in \ddot{\partial} K:= \partial K \cup \{\cp^0,\cp^1\}$. We claim that 
\labtequ{pes}{each component of $K \sm \Pi$ that meets $\cp \sm \bigcup_{q\in \ddot{\partial} K} P_q$ is a \pe.}
For this, let $C$ be such a component.
We will prove that $\partial C$ contains no point in $\partial K \sm \Pi$ and, and use this to prove  $|\partial C| = 2$. Indeed, suppose $\partial C$ contains a point $q \in \partial K \sm \Pi$. 
Then by \Lr{Sarcs}, $C\cup \{q\}$ contains a \arc{\cp}{q}  $A$. Let $q''\in \cp$ be the other endpoint of $A$.
Concatenating $A$ with the  \arc{q}{q'} $A_q$ chosen at the beginning of this proof, and applying \Lr{patharc} if needed, yields an arc witnessing the fact that the subarc $B$ of $\cp$ between $q',q''$ is bridged, and this is true even if $q'=q''$. But then $B$ can be used to extend $\cp_q$, contradicting the fact that $\cp_q$ is a maximal super-bridged subarc of $\cp$. This proves $\partial C \cap \partial K \sm \Pi= \emptyset$.

Next, we claim that $\partial C \subseteq \partial K \cup \Pi$. For if $p\in \partial C$, then since $C$ is closed in $X \sm (\partial K \cup \Pi)$, we have $p\in C$. By the local-connectedness of $X$, we can then find a connected open neighbourhood $O$ of $p$ avoiding $\partial K \cup \Pi$. This $O$ is a subset of $C$ by the definition of the latter, and contradicts the assumption that $p\in \partial C$, establishing our claim that $\partial C \subseteq \partial K \cup \Pi$. 

It is easy to see that $\partial C$ contains at least two point of $\Pi$ by contstruction. Suppose now $\partial C$ contains 3 distinct point $q,r,s$ of $\Pi$, and suppose \obda\ they appear on $\cp$ in that order. Note that $C$ must contain some point $x$ outside $\cp$ to be a component of $K \sm \Pi$. Applying \Lr{Sarcs} twice, once for $p=q$ and once for $p=s$, we can obtain arcs from $x$ to each of $q,s$. Combining these two arcs as above we obtain a \arc{q}{s} in $C \cup \{q,s\}$. Easily, this arc contains an arc $A$ \st\ the interior of $A$ does not meet $\cp$ and $r$ lies between the endpoints of $A$ on $\cp$. But $r$ is by definition an endpoint of a maximal super-bridged arc $P_q$, and $A$ contradicts its maximality since it bridges an interval extending $P_q$. 

Putting the above observations together proves that $|\partial C| = 2$, which implies that $C$ is open as $\partial C \cap C = \emptyset$, and it remains to prove that no open arc in $\cls{C}$ contains a point in $\partial C$. Suppose to the contrary that $A$ is such an arc containing a point $p\in \partial C$. Then $A$ contains a point $x$ outside $\cp \cap C$. By \Lr{Sarcs}, $x$ is connected to the other point $q\neq p$ of $\partial C$   by an arc $B$ in $C \cup \sgl{q}$. Concatenating the subarc of $A$ from $p$ to $x$ with $B$, and applying \Lr{patharc}, yields an \arc{p}{q} $Z$ in $\cls{C}$ Let $q'$ be the first point of $Z - p$ on $\cp$, which exists since $q\in \cp$. Then the subarc of $Z$ from $p$ to $q'$ bridges a non-trivial subarc of $\cp \cap C$. This subarc is thus bridged, and can be used to extend the super-bridged arc $P$ that has $p$ as an endpoint, which contradicts the maximality of $P$. This contradiction completes the proof of \eqref{pes}.

\medskip
Up to now we have dealt with a single component $K$, and despite proving \eqref{pes} it is still unclear that \pe s exist in $X$, since the set of components mentioned by \eqref{pes}  might be empty. We will now show that, although this can occur to some components of $X \sm \bigcup E$, if we choose $E$ appropriately then most components will have most of their \HM\ in \pe s. From now on we will follow a more global perspective, looking at all such components $K$ simultaneously, and applying the above contruction to each $K$.

For this we will also make use of the considerations of \Sr{ellH}. Recall that  by \Tr{HmeqL}, the quantity $\ch^G_\del$, defined as $\ch^G_\del:= \ell(E_\del) + \sum_{\stackrel{K \text{ is a component} }{\text{ of } X\sm E_\del}} diam(K)$, approximates $\Hm(X)$: \fe\ $\eps>0$, we can choose a finite edge-set $E$ \st\ 
\labtequ{sumdiam}{$\ell(E) + \sum diam(K) \geq \Hm(X) - \eps$.}

Applying the above contruction to each  component $K$ of $X \sm \bigcup E$, we fix a finite collection of arcs $\cp_K$ with $d(\cp_K^0,\cp_K^1)= diam(K)$ and a finite collection of maximal super-bridged subarcs $SB:= \{P_q \mid q\in \ddot{\partial} K, K \text{ is a component of }  X\sm E\}$. We are going to show that these arcs account for at most $\eps$ of $diam(K)$, from which will follow that the \pe s account for most of $diam(K)$. More precisely,
we claim that \fe\ component $K$ as above,
\labtequ{Hdiam}{ $\sum_{q\in \ddot{\partial} K} d(P_q^0,P_q^1) \leq \Hm(K) - diam(K)$,}
which is bounded above by \eqref{sumdiam}.

For this, let $P \in SB$, and recall that all but at most countably many points of $P$ lie in a bridged subarc in $\cb = \cb(K)$ for some component $K$. This means that $P \cap \bigcup \cb$ has full \HM\ $\Hm(P \cap \bigcup \cb)=\Hm(P)$, which implies that for an arbitrarily small $\eps_P>0$ we can find a finite subset $\cb'$ of $\cb$ \st\  $\Hm(\bigcup \cb')> \Hm(P) - \eps_P$. 

Recall that to each bridged arc $b\in \cb'$ \ti, by definition, a bridge $b^\cap$ with the same endpoints, which meets $\cp$ at its endpoints only. 
We are going to reroute $P$ through the bridges corresponding to the bridged arcs in $\cb'$ to obtain a new \arc{P^0}{P^1} $P'$: we replace each subarc $b$ of $P$ contained in $\cb'$ by its bridge $b^\cap$ to obtain a topological path from $P^0$ to $P^1$, and apply \Lr{patharc} to reduce it to the desired arc $P$. Note that $P'$ comprises finitely many subarcs of $P$ of total \HM\ at most $\eps_P$, joined by finitely many arcs contained in $\{b^\cap \mid b\in \cb'\}$. This implies
\labtequ{LA}{$\sum_{\{b^\cap \mid b\in \cb'\}} \ell(b^\cap) + \eps_P \geq \ell(P') \geq d(P^0,P^1)$.}
Suppose now that \eqref{Hdiam} is false, which means that \ti\ $\delta>0$ \st\ $\sum_{q\in \partial K} d(P_q^0,P_q^1) > \Hm(K) - diam(K) + \del$. Combining this with \eqref{LA} summed over all $A \in SB$, and using the fact that $\Hm(A)=\ell(A)$, yields
$$\sum_{\{b^\cap \mid b\in \cb'\}} \Hm(b^\cap) \geq \sum_{q\in \ddot{\partial} K} d(P_q^0,P_q^1) - \sum_{q\in \ddot{\partial}K} \eps_{P_q} > \Hm(K) - diam(K) + \del - \sum_{q\in \ddot{\partial}K} \eps_{P_q}.$$
Now note that the arcs $b^\cap$ above have only their endpoints on $\cp$, and all these arcs lie in $K$, and so $ \sum_{\{b^\cap \mid b\in \cb'\}} \Hm(b) + \Hm(\cp)\leq \Hm(K)$. As $\cp$ was chosen to be an arc joining two points at distance $diam(K)$, we have $\ell(\cp)\geq diam(K)$, and so we can rewrite the last inequality as $ \sum_{\{b^\cap \mid b\in \cb'\}} \Hm(b^\cap) \leq \Hm(K) - diam(K)$.
Combining this with the last inequality we obtain
 $$\Hm(K) - diam(K) \geq\sum_{\{b^\cap \mid b\in \cb'\}} \Hm(b^\cap)  > \Hm(K) - diam(K) + \del - \sum_{q\in \ddot{\partial}K} \eps_{P_q},$$
which means that $\sum_{q\in \ddot{\partial}K} \eps_{P_q}> \delta$. But we are allowed to choose the $\eps_{P_q}$ as small as we wish after fixing $\delta$. This contradiction proves \eqref{Hdiam}.

\medskip
Recall that, by \eqref{pes}, each component of $K \sm \Pi$ that meets $\cp \sm \bigcup_{q\in \ddot{\partial} K} P_q$ is a \pe; let $F=F(K)$ denote the set of these \pe s.
By the triangle inequality, we have
$$d(\cp_K^0,\cp_K^1) \leq \sum_{q\in \ddot{\partial} K} d(P_q^0,P_q^1) + \sum_{e\in F} d(e^0,e^1),$$
because $\cp$ is the concatenation of the arcs in $ \{P_q \mid q\in \partial K \cup \{\cp^0,\cp^1\}\} \cup (F \cap \cp) $ in an appropriate order. Moreover, we have $d(\cp_K^0,\cp_K^1)  = diam(K)$ because we chose $\cp$ so as to have this property. Using this in the above inequality and rearranging, we obtain
$$\sum_{e\in F} d(e^0,e^1) \geq  diam(K) - \sum_{q\in \partial K} d(P_q^0,P_q^1).$$ 
Plugging \eqref{Hdiam} into this yields
$$\sum_{e\in F(K)} d(e^0,e^1) \geq diam(K) -\Hm(K) + diam(K),$$ 
Summing this over all $K$, letting $PE_1:= \bigcup_{K \text{ is a component of }  X\sm \bigcup E} F(K)$, and using \eqref{sumdiam} now yields
$$\sum_{e\in PE_1} d(e^0,e^1)  \geq 2 \sum_\ck diam(K) - \sum_\ck \Hm(K) \geq 2 \Hm(X) -  2\ell(E) -2\eps - \sum_\ck \Hm(K),$$
where $\ck$ denotes the set of components $K$ of $X\sm \bigcup E$.
But as both $E$ and $\ck$ are finite (\Lr{finE}), we have $\Hm(X) = \ell(E) + \sum_\ck \Hm(K)= \Hm(E) + \sum_\ck \Hm(K)$, from which we deduce 
\labtequ{HF}{$\sum_{e\in PE_1} d(e^0,e^1) + \Hm(E) \geq  \Hm(X) - 2\eps.$}
Using \eqref{Hld} this implies
$$\sum_{e\in PE_1} \Hm(e) + \Hm(E) \geq  \Hm(X) - 2\eps.$$

Since every edge is also a \pe, the set $\cf_1:= E \cup PE_1$ is thus a set of pairwise disjoint \pe s of total \HM\ at least $\Hm(X) - 2\eps.$
\medskip

Let us put $\eps= \Hm(X)/M$ for some constant $M>1$ in the above construction. 
We thus managed to decompose the \gl\ continuum $X$ into a set of disjoint \pe s of total \HM\ $\frac{M-2}{M}\Hm(X)$ and a finite set of components with finite frontiers and total \HM\ $\frac{2}{M}\Hm(X)$. Note that each such component $C$ is itself a \gl\ continuum.  Thus we can repeat the whole process to decompose each such $C$ similarly into   \pe s of total \HM\ $\frac{M-2}{M}\Hm(C)$ and \gl\ components with finite frontiers in $X$. Iterating this proceedure recursively in \oo\ steps, we obtain a sequence \seq{\cf}, where $\cf_i= E_i \cup PE_i$, of \pe s \st\ the set $\cf:= \bigcup \cf_n$ is a set of 
pairwise disjoint \pe s of total \HM\ $\Hm( \bigcup \cf) = \sum_{f\in \cf} \Hm(f) = \Hm(X)$, where these equalities follow from the fact that \fe\ $\del>0$ \ti\ a finite subset $\cf'$ of $\cf$, namely those we have constructed after some step of our recursion,  with $\Hm( \bigcup \cf') = \sum_{f\in \cf'} \Hm(f)$ because they are pairwise disjoint, and $\Hm(X) \geq \Hm( \bigcup \cf') \geq \Hm(X)- \del$. 

Finaly, the total discrepancy $\sum_{f\in \cf} \delta(f) = \sum_{f\in \cf} \Hm(f) - \sum_{f\in \cf} d(f^0,f^1) $ of $\cf$ can be bounded using \eqref{HF} and the fact that $\Hm(X)=  \sum_{f\in \cf} \Hm(f)$ which we just proved. Recall that if $f$ is an edge, then we set $\delta(f)=0$. Thus the total discrepancy of $\cf_1$ is
\begin{eqnarray*}
\sum_{f\in \cf_1} \delta(f) = \sum_{f\in PE_1} \delta(f)
= \sum_{f\in PE_1} \Hm(f) - \sum_{f\in PE_1} d(f^0,f^1) \\
\leq \Hm(X)-\sum_{e\in E} \Hm(e) - \sum_{f\in PE_1} d(f^0,f^1) \\
\leq^{(\eqref{HF})} \Hm(X)-\sum_{e\in E} \Hm(e) + \Hm(E) - \Hm(X) + 2\eps =   2\eps = 2\Hm(X)/M, 
\end{eqnarray*}

By the same calculations, the total discrepancy of $\cf_i$ is\\ $\frac{2}{M}\sum_{C \text{ is one of the components of step } i} \Hm(C)$. The latter sum is bounded above by $\left( \frac{2}{M}\right)^{(i-1)} \Hm(X)$. Thus the total discrepancy of $\cf$ is $\Hm(X)\sum_i \left( \frac{2}{M}\right)^i$, which by the geometric series formula is a finite number tending to 0 as $M$ tends to infinity. Since we are allowed to choose any $M$ we want, this proves our claim that $\cf$ can be chosen with arbitrarily small total discrepancy.
\end{proof}

\medskip
{\bf Remark 1:} Given any finite set of points $P$ in $X$, we can choose $\cf$ so that $\cls{f}$ is disjoint from $P$ \fe\ $f\in \cf$: 
for if $f\in \cf$ contains a point of $P$ in its closure, we can apply again the above proceedure to $\cf$ instead of $X$, and keep doing so recursively in all \oo\ steps above, to split $f$ into an infinite set of \pe s none of which has a point or an endpoint in $P$.

\medskip
{\bf Remark 2:} Note that \fe\ $f\in \cf$, none of $f^0,f^1$ has an open neighbourhood contained in $f$. Indeed, $f^i$ was always chosen on an arc $\cp$ that has a non-trivial subarc outside $f$. This is important when applying \Lr{pseudo1}.

\medskip
We state the following corollary, obtained by combining \Tr{struct} with the above Remarks and Lemmas \ref{finEpe} and \ref{pseudo1}, in order to explicitely use it in \cite{bmg}.
\begin{corollary} \label{corstruct}
\Fe\ \gl\ continuum $X$ with \finhm, and every $\eps>0$, there is a finite set $\cf$ of pairwise disjoint \pe s of $X$ with the following properties
\begin{enumerate}
\item \label{cfi} $\sum_{f\in \cf} \Hm(f) > \Hm(X)  - \eps$;
\item \label{cfii} $\sum_{f\in \cf} \delta(f)< \eps$;
\item \label{cfiii} $X \sm \bigcup \cf$ has finitely many components, each of which is clopen in $X \sm \bigcup \cf$ and contains a point in $\overline{\cf}$;
\item \label{cfiv} \fe\ $f\in \cf$, and every \gseq\ \seq{G},  $G_n \cap \cls{f}$ is connected and contains a \pth{f^0}{f^1} for almost every $n$;
\item \label{cfv} $\cls{\bigcup \cf}$ avoids any prescribed point of $X$;
\item \label{cfvi} $\cf$ contains any prescribed finite edge-set.
\end{enumerate}
\end{corollary}

\section{The intrinsic metric} \label{secint}

The \defi{intrinsic metric} $\rho$ of a metric space $(X,d)$ is defined by 
$$\rho(x,y):= \inf {}_{R \text{ is an \arc{x}{y} }} \ell(R).$$ Note that if $X$ is a \gl\ continuum with $\Hm(X)<\infty$, then $\rho(x,y)$ is always finite by \Lr{cpc}.

Recall that, by \Tr{Xhg}, the completion $\widehat{G^\ell}$ of the union of any \gseq\ \seq{G} of $X$ is homeomorphic to $X$. Using our \pe\ decomposition theorem (\ref{struct}) we can now strengthen this by showing that $\widehat{G^\ell}$ does not depend on the choice of the sequence $(G_n)$:
\begin{theorem}
Let $(X,d)$ be a  \gl\ continuum with $\Hm(X)<\infty$. Then \fe\ \gseq\ \seq{G} of $X$, 
the metric $d_\ell$ of $\widehat{G^\ell}$ coincides with the intrinsic metric of $X$.
\end{theorem}
\begin{proof}
Let $x,y\in X$. It suffices to show that \fe\ $\eps>0$, and $n$ large enough, $G_n$ contains an \arc{x_n}{y_n} of length at most $\rho(x,y)+\eps$, where $d_\ell(x,x_n), d_\ell(y,y_n)< \eps$.

To show this, let $F$ be a finite set of pairwise disjoint \pe s \st\ 
\begin{enumerate}
\item \label{Fi} $\sum_{f\in F} \Hm(f)> \Hm(X)- \eps/2$ and 
\item \label{Fii} $\sum_{f\in F} \del(f)< \eps/2$,
\end{enumerate}
provided by \Tr{struct}. By Remark 1 after the proof of that theorem, we may assume that $x,y\not\in \bigcup F$.

Let $R$ be an \arc{x}{y} in $X$ with $\ell(R)< \rho(x,y)+\eps$, which exists by the definition of $\rho(x,y)$.
By \Lr{pseudo1}, almost every $G_n$ contains an \arc{f^0}{f^1} for every $f\in F$ and, similarly, \fe\ component $C$ of $X \sm \bigcup F$, the subgraph $G_n\cap C$ is connected by \eqref{comps2}. This allows us to replace the subarc of $R$ in any such $f$ or $C$ with an arc in $G_n$ with the same endpoints, except for an initial and final subarc of $R$ that may start in the interior of such a $C$, to obtain an arc $Q$ in $G_n$. We claim that $Q$ has the desired properties, i.e.\ the endpoints of $Q$ are \eps-close to $x,y$, and that $\ell(Q)< \ell(R) +\eps$.

Indeed, for the former claim, note that $\sum_{C\text{ is a component of }X \sm \bigcup F} \Hm(C) <\eps/2$ by \ref{Fi}, and that $d_\ell(x,Q^0)<\Hm(C)$ for the component $C$ containing both of $x,Q^0$.

For the later claim, we have $\sum_{C\text{ is a component of }X \sm \bigcup F} \ell(Q \cap C) \leq \eps/2$ by \ref{Fi}. Moreover, for $f\in F$ we have $\ell(Q \cap f) -\ell(R \cap f) \leq  \Hm(f) - d(f^0,f^1) =: \delta(f)$, and so 
$\sum_{f\in F} \ell(Q \cap f) -\ell(R \cap f) \leq \eps/2$ by \ref{Fii}. Putting these two inequalities together proves that $\ell(Q) -\ell(R)<\eps$ as desired.
\end{proof}

Combining this with \Tr{Xhg} and the fact that complete locally compact spaces are geodesic \cite[Theorem 2.5.23.]{Burago}, we obtain \Cr{corintr}.

\section{Convergence of effective resistances}

In this section we show that if $X$ is a \gl\ continuum with $\Hm(X)< \infty$, then we can associate to any pair of points $p,q\in X$ an effective resistance $R(p,q)$ as a limit of effective resistances $R_{G_n}(p_n,q_n)$ in a \gseq\ \seq{G} of $X$ between points $p_n,q_n$ converging to $p,q$. The interesting point here is that the limit $R(p,q)$ does not depend on the choice of these sequences. Let us first review some basics.

\subsection{Electrical network basics}

\newcommand{\ded}[1]{\ensuremath{\overrightarrow{#1}}}

An \defi{electrical network} is a graph $G$ endowed with an assignment of resistances $r: E \to \R_+$ to its edges. The set \defi{$\arE$} of \defi{directed edges} of \G\ is the set of ordered pairs $(x,y)$ such that $xy\in E$. Thus any edge $e$ of \g with endvertices $x,y$ corresponds to two elements of $\arE$, which we will denote by $\ded{xy}$ and $\ded{yx}$.
A \defi{\flo{p}{q}} of strength $I$ in \g is a function $i: \arE \to \R$ with the following properties
\begin{enumerate}
 \item $i(\ded{e^0 e^1}) = i(\ded{ e^1e^0 })$ \fe\ $e\in E$ ($i$ is antisymmetric);
\item \fe\ vertex $x\neq p,q$ we have  $\sum_{y\in  N(x)} i(\ded{xy}) = 0$, where \defi{N(x)} denotes the set of vertices sharing an edge with $x$ ($i$ satisfies Kirchhoff's node law outside $p,q$);
\item $\sum_{y\in  N(p)} i(\ded{py}) = I$  and $\sum_{y\in  N(q)} i(\ded{qy}) = -I$ ($i$ satisfies the boundary conditions at $p,q$).
\end{enumerate}

The \defi{effective resistance} $R_G(p,q)$ from a vertex $p$ to a vertex $q$ of \g is 
defined by 
$$R_G(p,q):= \inf_{i \text{ is a \flo{p}{q}\ of strength 1}} E(i),$$
 where the \defi{energy} $E(i)$ of $i$ is defined by $E(i) := \sum_{\ded{e}\in \arE}  i(\ded{e})^2 r(e)$. In fact, it is well-known that  this infimum is attained by a unique  \flo{p}{q}, called the corresponding \defi{electrical current}. 

The effective resistance satisfies the following property which justifies its name

\begin{lemma} \label{effres}
Let \g be an electrical network contained in an electrical network $H$ in such a way that there are exactly two vertices $p,q$ of \g connected to vertices of $H - G$ with edges. Then if $H'$ is obtained from $H$ by replacing $G$ with a \edge{p}{q} of resistance $R_G(p,q)$, then \fe\ two vertices $v,w$ of $H'$ we have $R_{H'}(v,w) =R_{H}(v,w) $. 
\end{lemma}
The proof of this follows easily from the definition of effective resistance. See e.g.\ \cite{LyonsBook} for details. 

\medskip
Any metric graph naturally gives rise to an electrical network by setting $r= \ell$, and we will assume this in the following section whenever talking about effective resistances in metric graphs.

\medskip
The following lemma shows that effective resistances can only decrease when contracting part of a metric graph. Define the total length $\ell(H)$ of a metric graph \g by $\ell(H) = \sum_{e\in E(H)} \ell(e)$.
\begin{lemma}\label{rcontr}	
Let \g be a finite metric graph and $G'$ be obtained from \g by contracting a connected subgraph $H$ to a vertex or contracting a subarc of an edge to a point. Then for any two vertices $p,q \in V(G)$, we have $  R_{G'}(\pi(p),\pi(q)) \in [R_G(p,q) - \ell(H), R_G(p,q)]$, where $\pi$ denotes the contraction map from $V(G)$ to $V(G')$.
\end{lemma}
\begin{proof}
Any flow on \g naturaly induces a flow on $G'$ by ingoring contracted edges. Thus the upper bound $  R_{G'}(\pi(p),\pi(q')) \leq R_G(p,q)$ is obvious from the definition of $r$. For the lower bound, note that the energy dissipated inside $H$ is at most $\ell(H)$ since it is well-known that no edge carries a flow greater than 1 in the $p$-$q$~current of strength 1.
\end{proof}

\subsection{Convergence of effective resistances for \gseq s}

The main result of this section is \Tr{rconv}, which we restate for the convenience of the reader
\newtheorem*{empth}{Theorem 1.1}
\begin{empth} 
Let $X$ be a \gl\ space with \finhm\ and $\seq{G}$ be a \gseq\ of $X$. 
Then \fe\ two sequences \seq{p}, \seq{q} with $p_n,q_n \in G_n$, each converging to a point in $X$, 
the effective resistance $R_{G_n}(p_n,q_n)$ converges to a value $R(x,y)$ independent of the choice of the sequences $(p_n), (q_n)$ and $(G_n)$.
\end{empth}

We will prove this by combining our structure theorem \Tr{struct} with the following fact, which can be thought of as a weaker version of \Tr{rconv} when $X$ happens to be a \pe.

\begin{lemma} \label{rpseudoe}
Let $f$ be a \pe\ of $X$ and $H$ a finite connected graph in $\overline f$ containing both $f^0,f^1$. Assign to each edge of $H$ a length (or resistance) equal to the length of the corresponding arc in $X$. Then 
$\Hm(f) \geq R_H(f^0,f^1) \geq 2d(f^0,f^1) - \Hm(f)$.
\end{lemma}
\begin{proof}
By definition, $H$ contains a \pth{f^0}{f^1}\ $P$. Since $\ell(P) \leq \Hm(f)$, the upper bound follows from the definition of $R$ since any flow in $P$ is a flow in $H$ and  $R_P(f^0,f^1) = \ell(P)$.

For the lower bound, assume \obda\ that our path $P$ is a shortest \pth{f^0}{f^1}\ in $H$. Since $f$ is a \pe, each of $f^0,f^1$ has degree 1 in $H$, for otherwise $H$ would contain an open arc containing one of $f^0,f^1$. 
We claim that $H$ can be contracted onto an \pth{f^0}{f^1} $P'$ with $\ell(P')\geq 2d(f^0,f^1) - \Hm(f)$, from which the assertion follows by \Lr{rcontr}. To perform this contraction, set $H_0=H, P_0= P$, and pick a component $C=C_0$ of $H - P$ if one exists. Let $P_C$ be the minimal subpath of $P$ containing all vertices of $P$ sending an edge to $C$ ($P_C$ might be a trivial path consisting of a vertex only). Then contract $P_C \cup C$ to a point. This contracts $H_0$ to a new graph $H_1$, and  $P_0$ to a \pth{f^0}{f^1}\ path $P_1$ in $H_1$. 
Repeat this process as often as needed, untill $ H_k= P_k$, and set $P':= P_k$.

To see that $\ell(P') $ satisfies the desired bound, recall that $P$ was a shortest \pth{f^0}{f^1}\ in $G$, and so $\ell(P_{C_i}) \leq \ell(C_i)$ for all components $C=C_i$ as above, for otherwise we could shortcut $P$ using a path in $C_i$. This means that 
$$\ell(P) - \ell(P') = \sum_i \ell(P_{C_i}) \leq \sum_i \ell(C_i) = \ell(H) - \ell(P) \leq \Hm(f) - \ell(P).$$ 

Thus we obtain $\ell(P') \geq 2\ell(P)- \Hm(f) $, and as $\ell(P)$ must be at least the distance of its endpoints, our claim follows.

\end{proof}

\begin{proof}[Proof of \Tr{rconv}]
Let $(p_n), (q_n)$ and $(G_n)$ be sequences as in the assertion. If we can prove that 
$R_{G_n}(p_n,q_n)$ converges, then the independence statement follows since we can combine any two candidate sequences into a new sequence by using alternating members.

\Fe\ $\eps>0$, we will find $i$ \leth\ $|R_{G_n}(p_n,q_n) - R_{G_m}(p_m,q_m)|< \eps$ \fe\ $n,m>i$, which implies convergence.

By \Tr{struct}, we can find a finite set $F$ of disjoint \pe s \st\ $\sum_{f\in F} \Hm(f) > \Hm(X) - \eps/4$ and $\sum_{f\in F} \delta(f)< \eps/4$. 

By \Lr{pseudo1}, we can find $i$ \leth\ $G_n$ cointais an \arc{f^0}{f^1}\ for every $f\in F$ and $n>i$ ---here we used Remark 2 after the proof of \Tr{struct}--- and by \eqref{comps2} we can even assume that $G_n \cap C$ is connected \fe\ component $C$ of $X \sm \bigcup F$. By Remark 1 after the proof of \Tr{struct}, we may assume that $F$ does not contain $\lim p_n$ or $\lim q_n$. Thus choosing $i$ a bit larger if necessary, we may assume that all $p_n$ lie in the same component of $X \sm \bigcup F$ for $n>i$, and similarly for the $q_n$.

Let $G'_n$ be the graph obtained from $G_n$ by replacing, for each $f\in F$, the subgraph $H_f:= G_n\cap f$ with an $f^0$-$f^1$~edge $e_f$ of length equal to the effective resistance $R_{H_f}(f^0,f^1)$. By \Lr{effres}, we have 

\labtequ{eqr}{$R_{G_n}(p_n,q_n) = R_{G'_n}(p_n,q_n)$.}

Now contract, for each component $C$ of $X \sm \bigcup F$, the subgraph $C\cap G'_n$ of $G'_n$ ---recall that we chose $i$ \leth\ this subgraph is always connected--- to a vertex, to obtain a new graph $G''_n$. Note that $G''_n$ is isomorphic to $G''_m$ as graphs for any $n,m>i$, only their edge-lengths can vary. We will use \Lr{rpseudoe} to deduce that they cannot vary too much. Indeed, applying that lemma to the graph $H_f$ yields the uniform lower bound $2d(f^0,f^1) - \Hm(f)$ for the length of the edge $e_f$ of $G''_n$ that replaced $f$. Let \g be the metric graph isomorphic to all $G''_n$ for $n>i$ in which each edge  $e_f$ is given length equal to this bound $2d(f^0,f^1) - \Hm(f)$. Note that each $G''_n$ can be contracted onto \g by contracting a subarc of each edge. Thus, by \Lr{rcontr}, we have $R_{G}(p,q) \in [R_{G''_n}(p,q)- \ell, R_{G''_n}(p,q)]$ for any two vertices $p,q$, where $\ell$ denotes the contracted length $\sum_{f\in F} R_{H_f}(f^0,f^1) - (2d(f^0,f^1) - \Hm(f) )$. 

By \Lr{rpseudoe}, we have 
$$\ell\leq \sum_{f\in F} \left( \Hm(f) - (2d(f^0,f^1) - \Hm(f) ) \right) = 2\sum_{f\in F}(\Hm(f) - d(f^0,f^1)) = 2 \sum_{f\in F} \delta(f)$$ 
by the definition of the discrepancy $\delta(f)$. Recall that we chose $F$ so that $\sum_{f\in F} \delta(f)< \eps/4$, and so we have proved that $R_{G}(p,q)$ differs from each $R_{G''_n}(p,q)$ by at most $\eps/4$.

Since we obtained the $G''_n$ from $G'_n$ by contracting connected subgraphs, applying \Lr{rcontr} again yields
$R_{G''_n}(p,q) \in [R_{G'_n}(p',q')- \ell', R_{G'_n}(p',q')]$ for any two points $p',q'$ mapped to $p,q$ by these contractions, where $\ell'= \sum \ell(C\cap G'_n)\leq \sum \Hm(C)$, the sum ranging over all components $C$ of $X \sm \bigcup F$. Recall that we chose $F$ so that $\sum_{f\in F} \Hm(f) > \Hm(X) - \eps/4$, and as $F$ and these components decompose $X$, we obtain $\ell' < \eps/4$.

The last two bounds combined imply that $|R_{G}(p,q)- R_{G'_n}(p',q')|< \eps/2$. Using our assumption that $p_n,q_n$ lie in the same component of $X \sm \bigcup F$ \fe\ $n>i$, we can apply this bound with $p'= p_n, q'=q_n$. Using \eqref{eqr} now proves our aim 
$|R_{G_n}(p_n,q_n) - R_{G_m}(p_m,q_m)|< \eps$.
\end{proof}

A \defi{local isometry} is a mapping $f$ from a metric space $(X,d)$ to a  metric space $(Y,g)$ \st\ \fe\ $x\in X$ \ti\ a neighbourhood $U\ni x$ \st\ $d(y,z) = g(f(y),f(z)$ \fe\ $y,z\in U$.   It is straightforward to check that local isometries preserve the lengths of topological paths. Since the edge-lengths of the graphs in our \gseq s are chosen to coinside with the lengths of the corresponding arcs of $X$, the values $R(x,y)$ in \Tr{rconv} is not affected by any change in the metric of $X$ that does not affect lengths of arcs. In particular, we have 
\begin{corollary} \label{lociso}
The values $R(p,q)$ in \Tr{rconv} are invariant under any homeomorphism that is a local isometry.
\end{corollary}
We remark that it has been proved that a  local isometry from a continuum to itself is a homeomorphism \cite{Calka}.

\section{Outlook}

We showed that one can define an effective resistance metric on \gl\ continua as a limit of effective resistances 
on a \gseq\ (\Tr{rconv}). Similarly, a \BM\ is constructed on  \gl\ continua as a limit of \BM s on a \gseq\ \cite{bmg}. The same approach should also yield extentions of other analytic objects to \gl\ continua, e.g.\ Laplacians, harmonic functions, Dirichlet forms, etc.

\medskip
Not every space with finite one-dimensional Hausdorff measure is \gl. Consider for example a space $X$ obtained from the real unit interval by attaching to the $i$th rational, in any enumeration of $\Q \cap [0,1]$, an arc of length $2^{-i}$. Thus $\Hm(X)=2$ but $X$ has no edge. However, one can ask

\begin{problem}
Does every space $X$ with $\Hm(X)<\infty$ admit a basis with finite frontiers? Can such an $X$ be obtained from a \gl\ space $X'$ with $\Hm(X')<\infty$ by metric contraction?
\end{problem}

In \Sr{gseqs} we constructed sequences of metric graphs that approximate our space $X$ well. It would be interesting to consider the converse question: 
\begin{question}
Under what conditions does a sequence of finite metric graphs converge (e.g.\ in the Gromov-Hausdorff sense \cite{Burago}) to a \gl\ space?
\end{question}

\medskip 

\comment{
1-dim Hausdorff measure and length are different as shown by Sierpinski example

\begin{problem}
If $X$ has finite length, does the metric constructed above have 1-dim Hausdorff measure equal to $\ell(X)$? 
\end{problem}
YES, EASY

\begin{problem}
If \finl\ does \ltp\ have  1-dim Hausdorff measure equal to $\ell$?
\end{problem}
YES, EASY

\begin{problem}
Let \seq{G}\ be a sequence of metric subgraphs of a complete \gl\ space $X$ that converges to $G$ in Gromov-Hausdorff sense and $\lim \ell(G_i)= \ell(G)< \infty$. Then \ti\ a non-expanding  homeomorphism from $X$ to $\hat{G}$. It is an isometry \iff\ $X$ is a length space. 
\end{problem}

\begin{problem}

\end{problem}
}

\note{
\section{Maybe..}

\begin{lemma} \label{EF}
Let $E$ and $F$ be two \des~s of a space $X$. Then $\bigcup E \sydi \bigcup F$ is totally disconnected.
\end{lemma} 

Let $E=e_1,\ldots$ be a \des. Contract each component of $X \sm \bigcup E_n$ to obtain $G'_n$. Note that $G'_{n+1}$ can be contracted to $G'_n$. Let $(\invg,d)$ be the metric inverse limit.

\begin{theorem}
\invg\ is homeomorhpic to $X$.
\end{theorem} 
\begin{proof} 

\end{proof}

\section{Lengths and geodesics}

Define the \defi{length} of a \gl\ space $X$ by $\ell(X):= \sup_{E \text{ is a \des\ of $X$}} \ell(E)$. 
(The infimum is always 0, as any edge has a Cantor set of Lebesque measure arbitrarily close to its length.) 

Our next result shows that $\ell(X)$ is attained by some \des.
\begin{lemma}\label{attain}	
Let $X$ be a \gl\ continuum. Then $X$ has a \des\ of length $\ell(X)$.
\end{lemma}
\begin{proof}

\end{proof}
}

\bibliographystyle{plain}
\bibliography{../collective}

\begin{thebibliography}{10}

\bibitem{BaPeBro}
M.~T. Barlow and E.~A. Perkins.
\newblock Brownian motion on the sierpinski gasket.
\newblock {\em Probability Theory and Related Fields}, 79(4):543--623, November
  1988.

\bibitem{BingPar}
R.~H. Bing.
\newblock Partitioning a set.
\newblock {\em Bulletin \ Am.\ Math.\ Soc.}, 55(12):1101--1111, December 1949.

\bibitem{Burago}
D.~Burago, I.~Burago, and S.~Ivanov.
\newblock {\em A Course in Metric Geometry}.
\newblock Am.\ Math.\ Soc., 2001.

\bibitem{Calka}
A.~Calka.
\newblock Local isometries of compact metric spaces.
\newblock {\em Proc.\ Am.\ Math.\ Soc.}, 85(4):643, August 1982.

\bibitem{ChrRiRoPla}
R.~Christian, R.~B. Richter, and B.~Rooney.
\newblock The planarity theorems of {MacLane} and whitney for graph-like
  continua.
\newblock {\em The Electronic Journal of Combinatorics}, 17(1):R12, May 2010.

\bibitem{RDsBanffSurvey}
R.~Diestel.
\newblock {Locally finite graphs with ends: A topological approach, I. Basic
  theory}.
\newblock In {\em {Infinite Graphs: Introductions, Connections, Surveys.
  Special issue of {\it Discrete Math.}}}, volume 311 (15), pages 1423--1447,
  2011.

\bibitem{ltop}
A.~Georgakopoulos.
\newblock Graph topologies induced by edge lengths.
\newblock In {\em {Infinite Graphs: Introductions, Connections, Surveys.
  Special issue of {\it Discrete Math.}}}, volume 311 (15), pages 1523--1542,
  2011.

\bibitem{lhom}
A.~Georgakopoulos.
\newblock {Cycle decompositions: from graphs to continua}.
\newblock {\em Advances in Mathematics}, 229(2):935--967, 2012.

\bibitem{gfm}
A.~Georgakopoulos, S.~Haeseler, M.~Keller, D.~Lenz, and R.~K. Wojciechowski.
\newblock Graphs of finite measure.
\newblock \href{http://arxiv.org/abs/1309.3501}{Preprint 2013}.

\bibitem{bmg}
A.~Georgakopoulos and K.~Kolesko.
\newblock Brownian motion on graph-like spaces.
\newblock Preprint 2013.

\bibitem{goldstein_random_1987}
S.~Goldstein.
\newblock Random walks and diffusions on fractals.
\newblock In H.~Kesten, editor, {\em Percolation Theory and Ergodic Theory of
  Infinite Particle Systems}, number~8 in The {IMA} Volumes in Mathematics and
  Its Applications, pages 121--129. Springer New York, January 1987.

\bibitem{ElemTop}
D.W. Hall and G.L. Spencer.
\newblock {\em Elementary Topology}.
\newblock John Wiley, New York 1955.

\bibitem{hambly_brownian_1997}
B.~M. Hambly.
\newblock Brownian motion on a random recursive sierpinski gasket.
\newblock {\em The Annals of Probability}, 25(3):1059--1102, July 1997.

\bibitem{Hattori}
T.~Hattori.
\newblock Asymptotically one-dimensional diffusions on scale-irregular gaskets.
\newblock {\em J. Math. Sci. Univ. Tokyo}, 4:229--278, 1997.

\bibitem{havlin_diffusion_1987}
S.~Havlin and D.~Ben-Avraham.
\newblock Diffusion in disordered media.
\newblock {\em Advances in Physics}, 36(6):695--798, 1987.

\bibitem{Kigami}
Jun Kigami.
\newblock {\em Analysis on fractals}.
\newblock Cambridge University Press, 2008.

\bibitem{KumBro}
T.~Kumagai.
\newblock Function spaces and stochastic processes on fractals.
\newblock In {\em Fractal geometry and stochastics III (C. Bandt et al.
  (eds.))}, volume~57 of {\em Progr. Probab.}, pages 221--234. Birkhauser,
  2004.

\bibitem{KumRec}
T.~Kumagai.
\newblock Recent developments of analysis on fractals.
\newblock In {\em Translations, Series 2}, volume 223, pages 81--95. Amer.\
  Math.\ Soc.\, 2008.

\bibitem{kusuoka_lecture_1993}
S.~Kusuoka.
\newblock Lecture on diffusion processes on nested fractals.
\newblock In {\em Statistical Mechanics and Fractals}, number 1567 in Lecture
  Notes in Mathematics, pages 39--98. Springer Berlin Heidelberg, 1993.

\bibitem{LyonsBook}
R.~Lyons and Y.~Peres.
\newblock {\em Probability on Trees and Networks}.
\newblock Cambridge University Press.
\newblock In preparation, current version available at {\small\tt
  http://mypage.iu.edu/\string~rdlyons/prbtree/prbtree.html}.

\bibitem{MenUnt}
K.~Menger.
\newblock Untersuchungen \"uber allgemeine metrik.
\newblock {\em Mathematische Annalen}, 100(1):75--163, December 1928.

\bibitem{MoiGri}
E.~E. Moise.
\newblock Grille decomposition and convexification theorems for compact metric
  locally connected continua.
\newblock {\em Bulletin \ Am.\ Math.\ Soc.}, 55(12):1111--1122, December 1949.

\bibitem{Nadler}
S.~B. Nadler.
\newblock {\em Continuum Theory: An Introduction}.
\newblock {CRC} Press, 1992.

\bibitem{strichartz1999analysis}
Robert~S Strichartz.
\newblock Analysis on fractals.
\newblock {\em Notices AMS}, 46(10):1199--1208.

\bibitem{ThomassenVellaContinua}
C.~Thomassen and A.~Vella.
\newblock Graph-like continua, augmenting arcs, and {M}enger's theorem.
\newblock {\em Combinatorica}, 29.
\newblock DOI: 10.1007/s00493-008-2342-9.

\end{thebibliography}
\end{document}